\documentclass[11pt,reqno]{amsart}
\usepackage{lmodern}
\usepackage{fancyhdr}
\usepackage{shuffle}
\usepackage{amssymb}
\usepackage{amsmath}
\usepackage{etoolbox}
\usepackage{thmtools}
\usepackage[usenames,dvipsnames]{color}
\usepackage[kerning=true]{microtype}
\usepackage[all,cmtip]{xy}
\usepackage{pgf,tikz}
\usepackage{wrapfig}
\usepackage{colortbl}
\usepackage{hhline}
\usepackage{caption}
\usepackage{subcaption}
\usepackage{eurosym}
\usetikzlibrary{decorations.markings,arrows,shapes,positioning,calc,shadows}
\tikzstyle arrowstyle=[scale=1]
\tikzstyle directed=[postaction={decorate,decoration={markings,
    mark=at position .5 with {\arrow[arrowstyle]{stealth}}}}]
\usetikzlibrary{shapes.geometric}
\usepackage{fullpage}
\usepackage{amsthm}
\usepackage{mathabx}
\usepackage[colorlinks=true, citecolor=cyan,urlcolor=blue,linktocpage=true,
	pagebackref,hyperindex=true,hypertexnames=true,hyperfigures=true]{hyperref} 
\usepackage[font=small,figurename=Fig.]{caption}

\tikzfading[name=fade l,left color=transparent!100,right color=transparent!0]
\tikzfading[name=fade r,right color=transparent!100,left color=transparent!0]
\tikzfading[name=fade d,bottom color=transparent!100,top color=transparent!0]
\tikzfading[name=fade u,top color=transparent!100,bottom color=transparent!0]

\newcommand\upnode[2][10pt]{
    \fill[white,path fading=fade d] (#2.north west) rectangle ($(#2.north east)+(0,-#1)$);
    \fill[white,path fading=fade l] (#2.south east) rectangle ($(#2.north east)+(-#1,0)$);
    \fill[white,path fading=fade r] (#2.south west) rectangle ($(#2.north west)+( #1,0)$);
}

\usepackage{genyoungtabtikz}\usepackage{amssymb}
\YFrench
\Yboxdim{1cm}
\Yfillcolour{green!80}

\definecolor{green}{rgb}{0.,0.6,0.} 
\newcommand\yred{\Yfillcolour{red}}
\newcommand\ylw{\Yfillcolour{yellow}}
\newcommand\ygreen{\Yfillcolour{green!80}}

\newcommand\bleu{\textcolor{blue}}

\newcommand\trois[2]{\begin{tikzpicture}[scale=.35]
\ylw
\tyng(0cm,0cm,2,1);
\ygreen
\tyng(0cm,0cm,#1,#2);
\end{tikzpicture}}

\declaretheoremstyle[
  spaceabove=6pt, spacebelow=6pt,
  headindent=0pt,
  headfont=\normalfont\bfseries,
  notefont=\mdseries, notebraces={(}{)},
    bodyfont=\itshape,
  postheadspace=1em]{mystyle}
\declaretheorem[numberwithin=section,style=mystyle]{theorem}
\declaretheorem[numberwithin=section,style=mystyle]{lemma}
\declaretheorem[numberwithin=section,style=mystyle]{proposition}
\declaretheorem[numberwithin=section,style=mystyle]{corollary}

\numberwithin{equation}{section}
\numberwithin{theorem}{section}
\numberwithin{figure}{section}
\numberwithin{table}{section}
\AfterEndEnvironment{proposition}{\noindent\ignorespaces}
\AfterEndEnvironment{lemma}{\noindent\ignorespaces}
\newtheorem{conjecture}{\bleu{Conjecture}}

\definecolor{vertfonce}{rgb}{0.,0.5,0.} 
\renewcommand\vert{\textcolor{vertfonce}}

\makeatletter
\renewenvironment{proof}[1][\proofname]{\par
  \pushQED{\qed}%
  \normalfont \topsep6\p@\@plus6\p@\relax
  \trivlist
  \item[\hskip\labelsep
        \scshape
    \vert{#1}\@addpunct{.}]\ignorespaces
}{%
  \popQED\endtrivlist\@endpefalse
}
\makeatother

\declaretheoremstyle[
spaceabove=3pt, spacebelow=3pt,
  headindent=0pt,
  headfont= \normalfont\bfseries,
  notefont=\mdseries, notebraces={(}{)},
    bodyfont=\normalfont,
  postheadspace=1em,
qed={\tiny\bleu{$\blacksquare$}}
]{probstyle}
\theoremstyle{definition}
\newtheorem{definition}[theorem]{\bleu{Definition}}

\renewcommand*\backref[1]{\ifx#1\relax \else (On page #1) \fi}
 
\newcommand\define[1]{\bleu{\bf #1}}

\newcommand\wrapcaption[1]{\captionsetup{font=tiny}
\caption{#1.}
\captionsetup{font=small}}

\DeclareMathOperator\youngdelta{\triangleleft\hskip-4pt\cdot\,}

\newcommand\ie{{\it i.e.}~}

\newcommand\area{\mathrm{area}}

\newcommand\jaune{\textcolor{yellow}}
\newcommand\Star[2]{\node at (#1+0.5,#2+0.4) {{\Huge$\boldsymbol\star$}};
				   \node at (#1+0.51,#2+0.41) {\jaune{\huge$\boldsymbol\star$}}; }

\newcommand{\Cat}{\mathcal{A}}

\newcommand{\des}{\mathrm{des}}
\newcommand\dinv{\mathrm{sim}}
\newcommand\simd{\mathrm{sim}}

\newcommand\Sim{\mathrm{Sim}}
\newcommand\Dyck{\mathcal{D}}

\newcommand\E{\mathcal{E}}

\newcommand\hook{\varepsilon}

\newcommand{\newatop}[2]{\genfrac{}{}{0pt}{}{#1}{#2}}
\newcommand\N{{\mathbb N}}

\newcommand{\qbinom}[2]{\genfrac{[}{]}{0pt}{}{#1}{#2}}

\renewcommand\S{{\mathbb S}}
\newcommand\SYT{{\textsc{syt}}}

\newcommand\q{\boldsymbol{q}}
\newcommand\x{\boldsymbol{x}}

\newcommand\Young{\mathbb{Y}}
\newcommand\YoungTriangle{\Young_\triangle}

\title[Triangular]{\large Combinatorics of Triangular Partitions}
\author{Fran\c{c}ois Bergeron}
\author{Mikhail Mazin}
\date{\today}

\begin{document} 
\begin{abstract}
The aim of this paper is to develop the combinatorics of constructions associated to what we call \emph{triangular partitions}. As introduced in~\cite{2102.07931}, these are the partitions whose cells are those lying below the line joining points $(r,0)$ and $(0,s)$, for any given positive reals $r$ and $s$. Classical notions such as Dyck paths and parking functions are naturally generalized by considering the set of partitions included in a given triangular partition. One of our striking results is that the restriction of the Young lattice to triangular partition has a planar Hasse diagram, with many nice properties. It follows that we may generalize the ``first-return'' recurrence, for the enumeration of classical Dyck paths, to the enumeration of all partitions contained in a fixed triangular one.   
\end{abstract}
\maketitle

{ \setcounter{tocdepth}{1}\parskip=0pt\footnotesize \tableofcontents}
\parskip=8pt  
\parindent=20pt


\section*{Introduction}
The last 30 years have seen increasing and interesting interactions between the combinatorics of generalized Dyck paths and parking functions, Macdonald polynomials and operators, diagonal coinvariant spaces, Hilbert schemes of points, Khovanov-Rozansky homology of $(m,n)$-torus links, Double affine Hecke algebra, Elliptic Hall algebra, and Superspaces of bosons and fermions. The aim of this paper is to further the combinatorial study of the most general extension of the notion of Dyck paths that is involved in this context. 

For the purpose of describing our overall setup, it is handy to consider classical Dyck paths as partitions $\alpha\subseteq \tau_{(n+1,n)}$ contained in the staircase partition $\tau_{(n+1,n)}:=(n-1,n-2,\ldots,1,0)$. These are well known to be enumerated by the Catalan numbers. From this point of view, parking functions are simply standard tableaux of skew shape $(\alpha+1^n)/\alpha$, for any such $\alpha$. Recall that they number $(n+1)^{n-1}$ in total. The  generalization we consider here is obtained by changing the overall partition $\tau_{(n+1,n)}$ to other suitably chosen partitions $\tau$. For instance, one may consider the more the general staircase $\tau_{(jn+1,n)}:=(j(n-1),j(n-2),\ldots,k)$, for some $j\geq 1$. In that case, it is well known that the number of partitions included in $\tau_{(jn+1,n)}$ is given by the Fuss-Catalan number $1/(jn+1)\binom{(j+1)n}{n}$. The corresponding number of parking functions is  given by the simple formula $(jn+1)^{n-1}$.  In the historical sequence of generalizations, \emph{Rational Catalan Combinatorics} (see~\cite{ArmstrongLoehr}) comes next in line, with an overall partition consisting in the set of cells lying below the diagonal of a $(k\times n)$-rectangle, with $k$ and $n$ coprime. The terminology ``rational'' alludes to $k/n$ being a rational number. Removing the coprimality condition, one gets \emph{Rectangular Catalan Combinatorics} (see~\cite{MR3682396}). 

The final step in the hierarchy, introduced in~\cite{2102.07931}, corresponds to choosing any overall ``triangular partition''. These are the partitions whose cells are those lying below the line joining points $(r,0)$ and $(0,s)$, for any given positive reals $r$ and $s$. The resulting partition\footnote{These are defined in Section \ref{sec_triangular}.}, are here denoted by $\tau_{rs}$. As mentioned previously, the rational case corresponds to choosing $r=k$ and $s=n$, relatively prime integers, and it is well known that the number of partitions contained in $\tau_{kn}$ is given by the formula $\frac{1}{k+n}\binom{k+n}{n}$,
whilst the number of associated parking functions is ${k^{n-1}}$.
These enumerations become a bit trickier when one drops the coprimality requirement\footnote{This is the ``rectangular'' case.}, taking the respective forms
\begin{align*}
 \bleu{\sum_{\mu\vdash d} \frac{1}{z_\mu} \prod_{j\in\mu} \frac{1}{a+b} \binom{j(a+b)}{jb}},\qquad {\rm and}\qquad
 \bleu{\sum_{\mu\vdash d} \frac{n!}{z_\mu}\prod_{j\in\mu}\frac{(ja)^{jb}}{a\,(jb)!}},
\end{align*} 
where $a:=k/d$, $b:=n/d$, for $d=\gcd(k,n)$;  and writing $j\in\mu$ for $j$ being a part of $\mu$. Recall here that $z_\mu:=\prod_i i^{d_i}\, d_i!$,  with $d_i$ being the number of $i$-size parts of $\mu$. 

The extension to any triangular partitions, when $r$ and $s$ are any positive real numbers, is motivated by the results in~\cite{2102.07931}. The purpose of this paper is to explore different aspects of this most general framework.
We start with a discussion of properties and intrinsic characterization triangular partitions, including interesting aspects of the dominance and containment order. We then obtain general recurrence for the (weighted) enumeration of triangular Dyck paths and associated parking functions (see Propositions \ref{proposition_Delta_recurrence} and \ref{proposition_Parking_recurrence}).

\section{Triangular partitions}\label{sec_triangular}
\begin{wrapfigure}{r}{0.32\textwidth}
\vskip-15pt
 \includegraphics[width=0.31\textwidth]{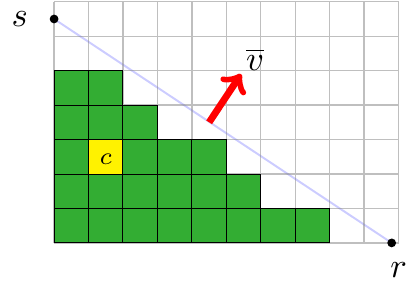}
 \vskip-10pt
\wrapcaption{Triangular partition}
\label{fig_triangle} 
\end{wrapfigure}
A partition $\tau=\tau_1\tau_2\cdots \tau_k$ is said to be \define{triangular} if there exists positive real numbers $r$ and $s$ such that 
$$\bleu{\tau_j} =  \bleu{\big\lfloor r-{j\,r}/{s}\big\rfloor},\hskip3.5cm \null$$
with $j$ running over integers that are less or equal to $s$. 
In other words, the cells  $c=(i,j)\in \N^2$ of the triangular partition $\tau_{rs}$ are those that lie below the line joining $(0,s)$ to $(r,0)$, \ie
	$$\bleu{(i,j)\in\tau_{r,s}}\qquad {\rm iff}\qquad \bleu{\frac{i+1}{r}+\frac{j+1}{s}}\leq \bleu{1}.\hskip3cm \null$$
We say that this line \define{cuts off} $\tau$.
Clearly, the conjugate of a triangular partition is triangular. 
Triangular partitions have been considered under the name ``plane corner cuts''  in~\cite{MR1692984},  and enumerated in~\cite{MR1692980}. Denoting by $t_N$ the number of triangular partitions of size $N$, small values of the sequence $\{t_N\}_N$ are:
$$1,1,2,3,4,6,7,8,10,12,13,16,16,18,20,23,\ldots$$
\autoref{TableTriangularPartitions} displays all triangular partitions of $N\leq 6$, with $\varepsilon$ denoting the ``empty'' partition.
\begin{table}[ht]
\ylw
 \Yboxdim{.3cm}
\begin{align*}
&\varepsilon,\quad \yng(1),\quad \yng(2),\yng(1,1),\quad \yng(3),\yng(2,1),\yng(1,1,1),\quad \yng(4),\yng(3,1),\yng(2,1,1),\yng(1,1,1,1),\\
&\yng(5),\yng(4,1),\yng(3,2),\yng(2,2,1),\yng(2,1,1,1),\yng(1,1,1,1,1),\quad \yng(6),\yng(5,1),\yng(4,2),\yng(3,2,1),\yng(2,2,1,1),\yng(2,1,1,1,1),\yng(1,1,1,1,1,1).
\end{align*}
\caption{All triangular partitions, for $n\leq 6$.}\label{TableTriangularPartitions}
\end{table}

Observe that many different pairs $(r,s)$ may give rise to the same triangular partition $\tau$. Indeed, as illustrated in \autoref{Fig_Trig_Droites}, we have $\tau=\tau_{rs}=\tau_{r's'}$ whenever there are no positive integer coordinate points lying between the lines ${x}/{r}+{y}/{s}=1$ and  ${x}/{r'}+{y}/{s'}=1$. 
In particular, $\tau_{rn}=\tau_{kn,n}$ for all $kn\leq r\leq kn+1$ and $k,n\in\N$.

We say that the line $x/r+y/s=1$ \define{touches}  the cell $(i,j)$ (from above) if it contains the north-east corner of the cell, \ie $(i+1)/r+(j+1)/s=1$. 
If the line $x/r+y/s=1$ touches no cells, then $\tau_{r's'}=\tau_{rs}$ for any $(r',s')$ close enough to $(r,s)$. If the line $x/r+y/s=1$ touches a cell $(i,j)$, then taking $r'=r-\epsilon$ and $s'=s-\epsilon$, for a sufficiently small positive $\epsilon$, one may ``remove'' the cell $(i,j)$ from the diagram of the partition. On the other hand, taking $r'=r+\epsilon$ and $s'=s+\epsilon$ we get a line cutting off the same partition while touching no cells. Summing up, for every triangular partition $\tau$ there is a pair $(r,s)$ such that the line $x/r+y/s=1$ cuts off $\tau$, without touching any cell.

Let us say that $\tau$ is \define{integral}, if there exist $k$ and $n$ in $\N$ such $\tau=\tau_{kn}$. As mentioned previously, these are the triangular partitions which give rise to the previous context of ``rectangular'' Catalan combinatorics; and adding the requirement that $\gcd(k,n)=1$ that of ``rational'' Catalan combinatorics. Clearly conjugation preserves integrality. For size $N\leq 10$,  the non-integral triangular partitions are the following:
\begin{align*}
\{&2111, 221, 32, 41, \\
&21111, 51,  \\
&211111, 3211, 421, 61, \\
&2111111, 221111, 22211, 53, 62,  71, \\
&21111111, 2211111, 3321, 432, 72, 81, \\
&211111111, 22111111, 322111, 33211, 532, 631, 82, 91\}.
\end{align*}
As $N$ grows, non-integral triangular partitions become preponderant. For instance, among all triangular partitions of size at most $55$, more than $87.5\%$ of them are non-integral. Thus triangular partitions significantly extend the field of study.

A \define{slope vector} for a triangular partition $\tau$ is a positive coordinate vector which is orthogonal to one of the lines that cut off $\tau$. If need be, we may normalize slope vectors so that their coordinates sum up to $1$. 
For any two lines cutting of the same triangular partition $\tau$, say with respective slope vectors $(t_1,1-t_1)$ and $(t_2,1-t_2)$, there exist a line with slope vector $(t,1-t)$ that also cuts of $\tau$ for any $t_1<t<t_2$. By definition, a triangular partition affords (infinitely many) slope vectors. Hence, it follows that:
\begin{lemma}
The closure of the set of all slope vectors of a triangular partition $\tau$ forms a convex cone $C_\tau$, of the form
	 \begin{equation} 
	 	\bleu{C_\tau} = \bleu{\{\lambda(t,1-t)\, |\, t^-_\tau \leq t \leq t^+_\tau,\, \lambda\in\mathbb{R}_+\}},
	\end{equation}
for some $0\leq t^-_\tau< t^+_\tau\leq 1$. Moreover, $\tau$ is uniquely characterized by its size and the pair $(t^-_\tau, t^+_\tau)$.
\end{lemma}
We will further discuss this unicity below, but first we show how to explicitly calculate $t^-_\tau$ and $t^+_\tau$. As we will see, this also gives a direct characterization of triangular partitions, without having to exhibit a cutting line.
\begin{figure}[ht]
 \includegraphics[scale=.66]{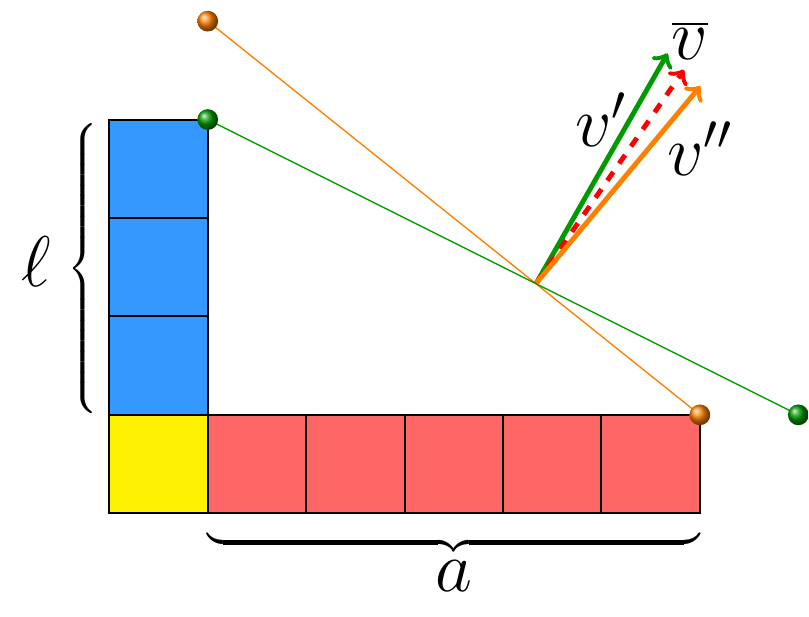}
  \caption{The two extreme slope vectors (here upscaled) of a hook shape $(a\,|\,\ell)$.}
  \label{fig_hook_slopes}
  \end{figure}
  
For a cell $c$ of a partition $\mu$, recall that the \define{hook} of $c$ is the partition of shape $(a+1,1^\ell)$, with $a=a_\mu(c)$ and $\ell=\ell_\mu(c)$ standing  for the \define{arm} and \define{leg} of $c=(i,j)$ of $\mu$. Recall that $a_\mu(c):=\mu_j-i$ (resp. $\ell_\mu(c):=\mu'_i-j$) is the number of cells of $\mu$ that sit to the right of (resp. above) $c$ in the same row (resp. column). The \define{hook length}\index{cell!hook length}  of  cell $c$ in $\mu$, is defined to be \define{$\hook(c)=\hook_\mu(c):= 1+a_\mu(c)+\ell_\mu(c)$}. The ``Frobenius notation'' for such a hook is $(a\,|\,\ell)$.

For any given cell $c$ of a partition $\mu$, we consider the vectors $(t'(c),1-t'(c))$ and $(t''(c),1-t''(c))$  respectively orthogonal to the lines: 
\begin{itemize}
\item  which joins the vertices $(1,\ell(c)+1)$ and $(a(c)+2,1)$, and 
\item which joins the vertices $(1,\ell(c)+2)$ and $(a(c)+1,1)$.
\end{itemize}
An instance is illustrated  in \autoref{fig_hook_slopes}.
 It is easy to check that
\begin{equation}\label{def_slope_t}
\bleu{t'_\mu(c)}:= \bleu{{\ell(c)}/{\hook(c)}},\qquad {\rm and}\qquad  \bleu{t''_\mu(c)} = \bleu{{(\ell(c)+1)}/{\hook(c)}}.
\end{equation}
When $\tau$ is a triangular partition, all of its slope vectors $v=(t,1-t)$ must be such that  $t'_\tau(c)<t<t''_\tau(c)$. Indeed, all the cells appearing in the hook of $c$ must lie below (up to parallel shift) any line that cuts off $\tau$, with $t'_\tau(c)$ and $t''_\tau(c)$ giving extreme bounds. Otherwise, an extra cell would have to lie in $\tau$, either at the end of the arm or the leg of $c$. 

Vice versa, if the condition $t'_\tau(c)<t<t''_\tau(c)$ is satisfied for every cell $c$ in $\tau$, then $(t,1-t)$ is a slope vector of $\tau$. Indeed, consider the lowest line $(r,s)$ perpendicular to $(t,1-t)$ and such that $\tau\subset \tau_{rs}$. It follows that the line $(r,s)$ touches a cell of $c'$ of $\tau$. Suppose that there is a cell $c''$ that fits under the line $(r,s)$, but such that $c''\notin \tau$. Without loss of generality, we can assume that $c'$ lies to the west of $c''$. Let $c\in\tau$ be the cell in the same column as $c'$ and the same row as $c''$. Then $t\ge t''_\tau(c)$ which is a contradiction. Hence there is no such cell $c''$ and $\tau=\tau_{rs}$ We have thus shown the following:

\begin{lemma}\label{lemma_slope}
A partition $\tau$ is triangular if and only if $t^-_\tau< t^+_\tau$, with 
\begin{equation}\bleu{t^-_\tau:=\max_{c\in\tau} \frac{\ell(c)}{a(c)+\ell(c)+1}},\qquad {\rm and}\qquad 
        \bleu{t^+_\tau:=\min_{c\in\tau}\frac{\ell(c)+1}{a(c)+\ell(c)+1}}.
 \end{equation}
Furthermore, if $t^-_\tau< t^+_\tau$ then $(t,1-t)$ is a slope vector of $\tau$ if and only if $t^-_\tau<t< t^+_\tau.$
\end{lemma}
As illustrated in \autoref{Fig_Trig_Droites}, this gives an explicit description for the extreme rays of the cone $C_\tau$ which does not rely on an explicit knowledge of a line cutting off $\tau$. Setting $t_\tau:=(t^-_\tau+t^+_\tau)/2$,  we may consider the vector $v_\tau:=(t_\tau,1-t_\tau)$ to be a ``standard'' \define{slope vector} for $\tau$.
\begin{figure}[ht]
 \includegraphics[scale=.75]{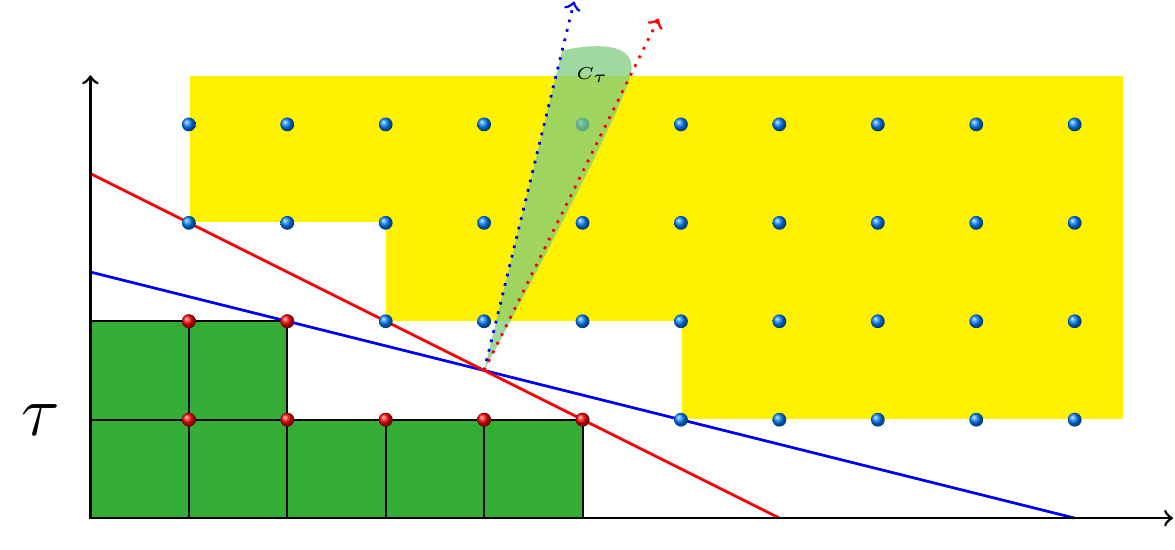}
\caption{The cone of a triangular partition.}\label{Fig_Trig_Droites}
\end{figure}

Two triangular partitions are said to be \define{similar}\footnote{Careful, this is not intended to be an equivalence relation.} if they share a slope vector. For any triangular partition $\tau$, it is easy to check that all partitions associated to $\tau$-\define{shadows} of a cell $c=(i,j)$:
  	$$\bleu{{\rm sh}_c(\tau)}:=\bleu{\{(x,y)\,|\, (x,y)\in \tau,\ {\rm with}\ x\geq i\ {\rm and}\ y\geq j\}},$$
are similar to $\tau$. For any fixed $t\geq 0$,  all the partitions with cell sets: 
   	$$\bleu{\{ (i,j)\,|\, i\,t+j\,(1-t)< d\}},\qquad \bleu{0\leq d < \infty},$$ 
are clearly similar, since they share the slope vector $(t,1-t)$. These partitions are nested as $d$ grows. Furthermore, if one chooses $t$ so that $t/(1-t)$ is irrational, cells are added one at a time as $d$ grows. We get the following

\begin{corollary}\label{corollary: tripar of size n}
For a fixed size $n$ there are finitely many values 
\begin{equation*}
0=t_0<t_1<\ldots<t_k<1=t_{k+1}
\end{equation*}
 such that for every $i\in\{0,1,\ldots,k+1\}$ no triangular partition of size $n$ admits a slope vector $(t_i,1-t_i)$.  Furthermore, for every $i\in\{0,1,\ldots,k\}$ there is a unique triangular partition $\tau_i$ of size $n$ such that $t^-_{\tau_i}=t_i$ and $t^+_{\tau_i}=t_{i+1}$. 
\end{corollary}

\subsection*{Note}
Part of the study of triangular partitions may be coined in terms of  ``sturmian'' words. Indeed, these give explicit descriptions of discrete lines in the plane. Their factors, also known as ``mechanical'' word, describe segments of discrete lines. The upper bound (cells closest to a cutting line) of  triangular partitions correspond to such segment. However, this approach does not emphasize the number of cells lying below the line. For more on sturmian and mechanical words, see~\cite[Chapter 2]{MR1905123}.  

\section{Moduli space of lines}\label{section_lines}

Note that a line $x/r+y/s=1$ touches a cell $(i,j)=(a-1,b-1)$ if and only if
\begin{equation*}
a/r+b/s=1\quad {\rm iff}\quad
as+br=rs\quad {\rm iff}\quad
(r-a)(s-b)=ab.
\end{equation*}
Therefore the positive $(r,s)$-quadrant $\mathcal{Q}:=(\mathbb{R}_{>0})^2$ is can be decomposed into regions by (positive branches of) hyperbolas with equations as above, one for each cell. Thus $\mathcal{Q}$ may be considered as a \define{moduli space of lines}. In this moduli space, the {hyperbola} 
$$\mathcal{H}_{ab}:=\{(r,s)\in \mathcal{Q} \, |\, (r-a)(s-b)=ab\},$$
associated to a cell $(i,j)=(a-1,b-1)$, separates $\mathcal{Q} $ according to whether the cell $(i,j)$ fits below the line $(r,s)$, is touched by the line $(r,s)$, or does not fit under it. These possibilities respectively correspond to 
\begin{equation*}
(r-a)(s-b)\geq ab,\qquad
(r-a)(s-b)=ab,\qquad {\rm and} ,\qquad (r-a)(s-b)< ab.
\end{equation*}
Thus the cells occurring in a triangular partition specified by $(r,s)$ are in bijection with the hyperbolas separating $(r,s)$ from the origin. In other words, as one crosses the hyperbola $\mathcal{H}_{ab}$, going towards the origin, the cell $(i,j)=(a-1,b-1)$ gets removed from the partition. For this reason it is natural to say that $\mathcal{H}_{ab}$ is the \define{hyperbolic wall} associated to $(a,b)$.

\begin{lemma}\label{lemma_bijection}
There is a bijective natural  map\footnote{Illustrated in \autoref{Fig_Triangular_Young_dual}.} $\rho:\mathcal{C}\to \mathcal{T}$, from the set $\mathcal{C}$ of connected components of the complement:
\begin{equation*}
 \overline{\mathcal{Q} }:=\mathcal{Q}\, \setminus \bigcup\limits_{(a,b)\in \mathbb{N}^2} \mathcal{H}_{ab},
\end{equation*}
to the set $\mathcal{T}$ of triangular partitions.  
\end{lemma}
\begin{figure}
\centering
\begin{subfigure}[t]{3in}
\centering
 \begin{tikzpicture}
	\node at (0,0) {\includegraphics[scale=.6]{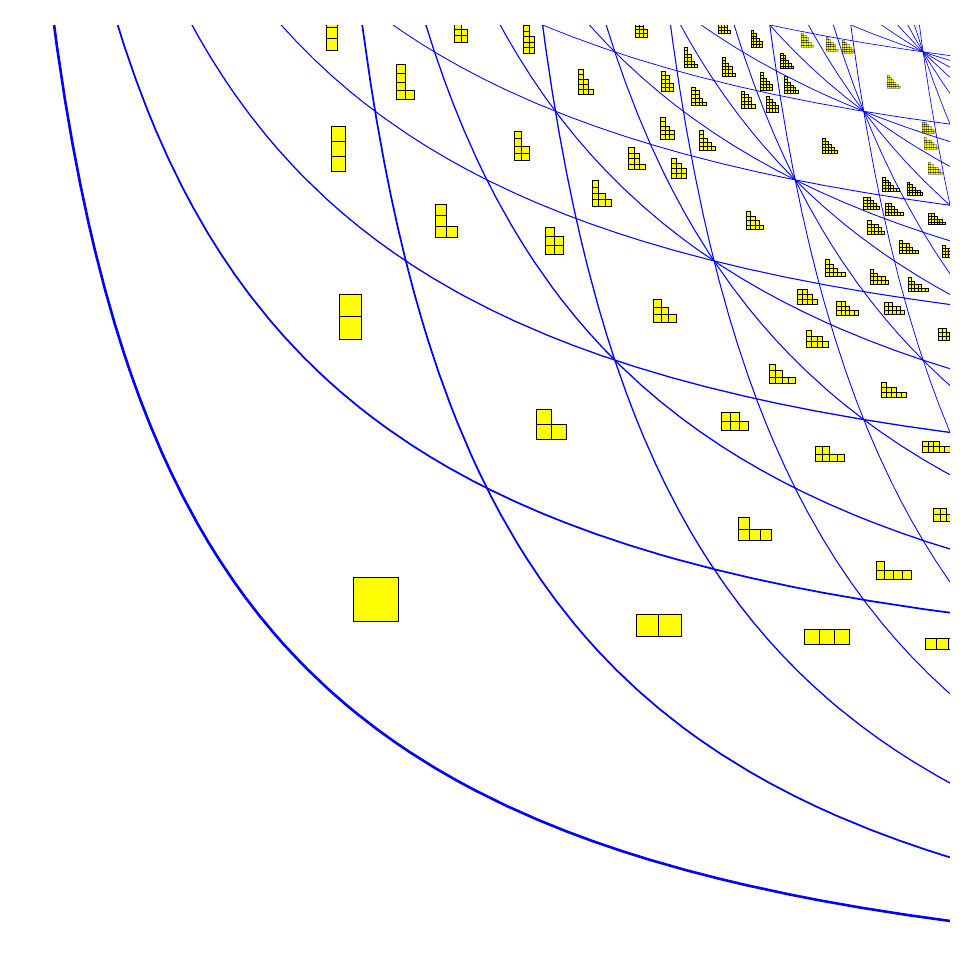}};
	 \node at (-1.4,-1.4) {$\varepsilon$};
        \draw[->] (-3,-3) -- (3,-3) node[below] {$r$};
  	\draw[->] (-3,-3) -- (-3,3) node[left] {$s$};
\end{tikzpicture}
	\caption{With regions labeled}\label{fig:1a}
\end{subfigure}
\begin{subfigure}[t]{3in}
\centering
 \begin{tikzpicture}
	   \node at (0,0) {\includegraphics[scale=.14]{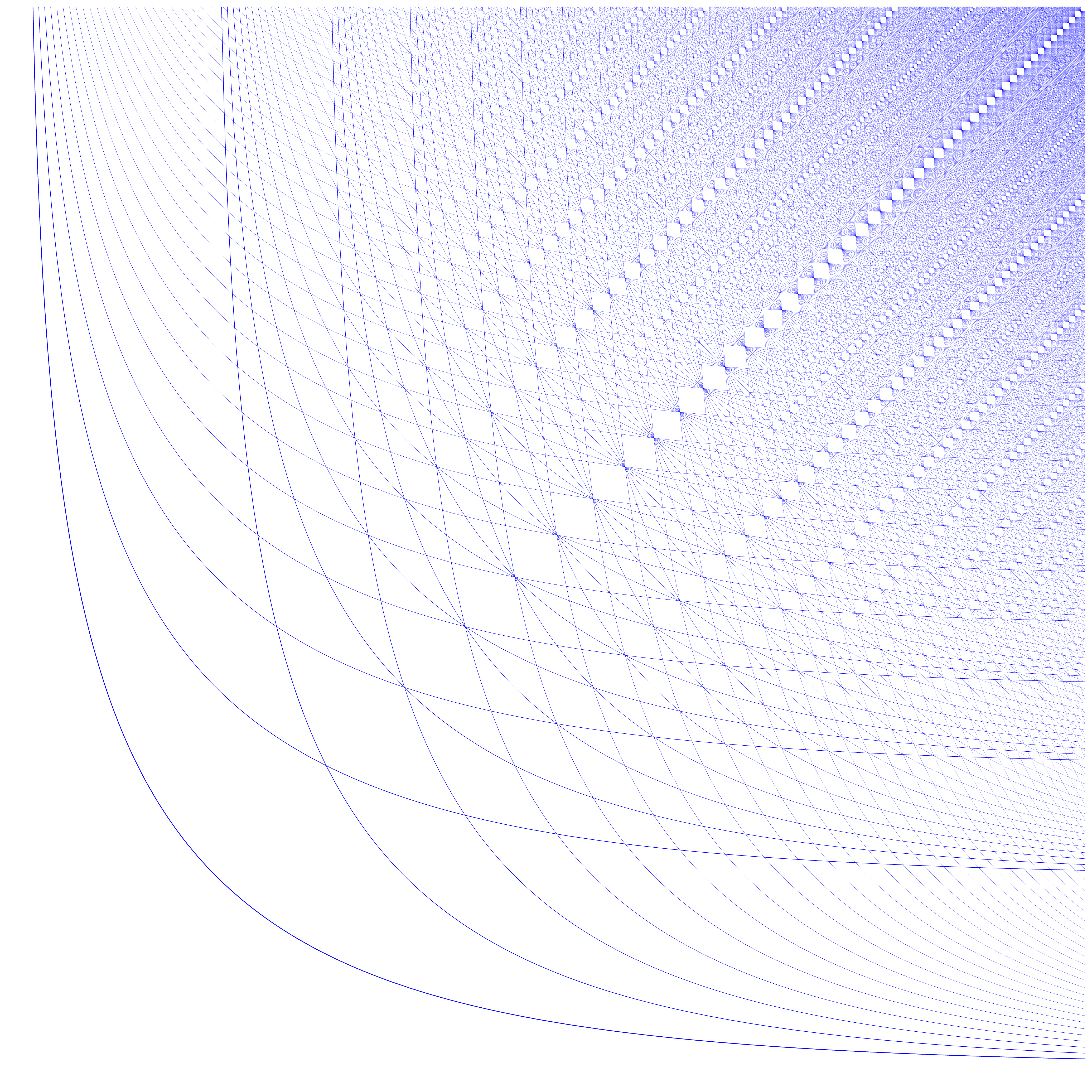}};
        \draw[->] (-3,-3) -- (3,-3) node[below] {$r$};
  	\draw[->] (-3,-3) -- (-3,3) node[left] {$s$};
\end{tikzpicture}
 	\caption{Larger portion}\label{fig:1b}
\end{subfigure}
 \caption{Triangular partition regions, in logarithmic scale.}
 	\label{Fig_Triangular_Young_dual}
\end{figure}
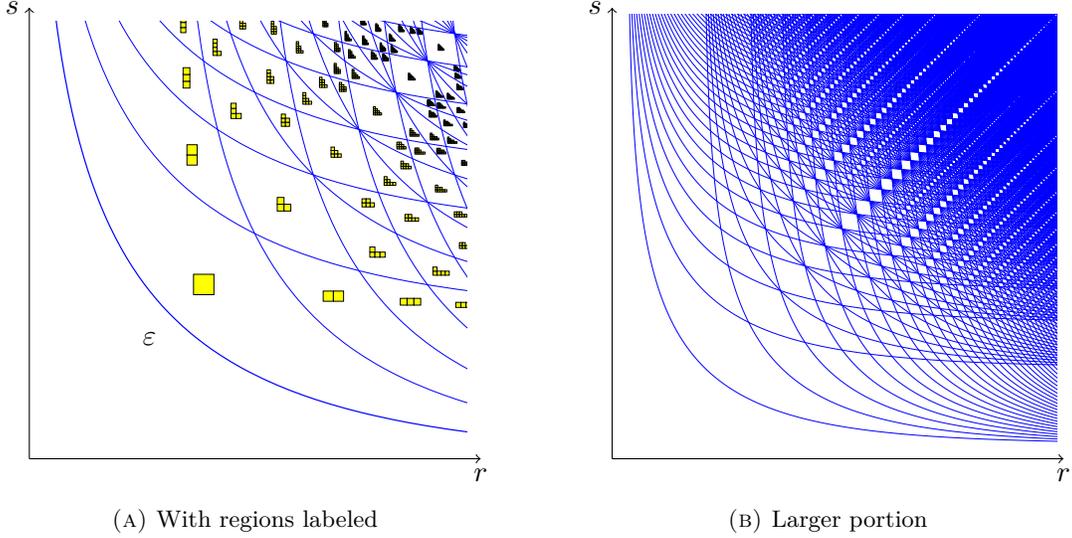
\begin{proof}
Consider the map that sends the $(r,s)$-line in $\overline{\mathcal{Q} }$ to the partition $\tau_{rs}$. 
As observed above, this is a locally constant map from  $\overline{\mathcal{Q} }$ to the set of triangular partitions. Therefore, it is well defined on connected components. The surjectivity is immediate, since by definition a triangular partition $\tau$ is cut off by some line $x/r+y/s=1$, and we have observed that $(r,s)$ may be chosen so that this line contains no points of $\N^2$.
To prove injectivity one needs to show that whenever two points of  $\overline{\mathcal{Q} }$ define the same triangular partition, then they must lie in the same connected component. We consider two cases. 

\begin{figure}
 \includegraphics[scale=.8]{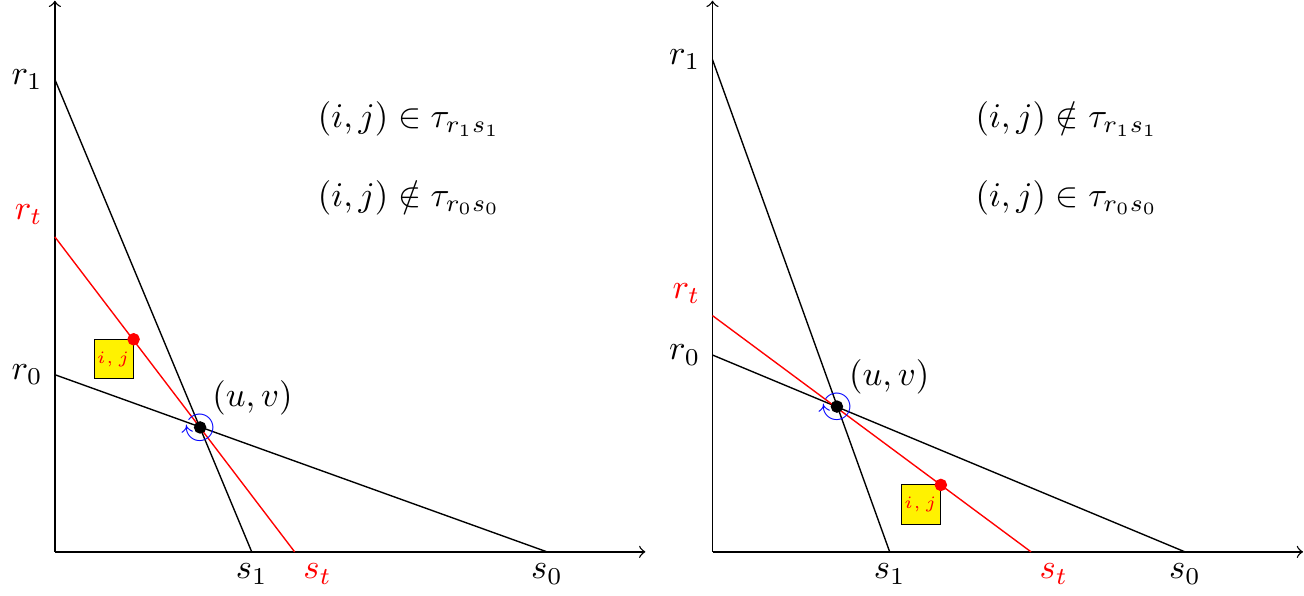}
\caption{Illustration of \autoref{lemma_bijection}.}
\label{figure: lemma 1 case 1}    
\end{figure}

First, when the lines associated to $(r_0,s_0)$ and $(r_1,s_1)$ do not intersect in the positive quadrant, then one may assume without loss of generality that $r_0\le r_1$ and $s_0\le s_1$. Let us show that the segment $(r_t,s_t):=(r_0+t(r_1-r_0),s_0+t(s_1-s_0))$, connecting these two points of ${\mathcal{Q} }$, lies entirely within the same connected component of $\overline{\mathcal{Q} }$. Indeed, if the segment $(r_t,s_t)_{0\leq t\leq 1}$ would cross one of the hyperbolic walls $\mathcal{H}_{ab}$, then there would exist some $0<t<1$ such that the line ${x}/{r_t}+{y}/{s_t}=1$ passes through $(a,b)$. The partition corresponding to $(r_1,s_1)$ would thus contain the cell $(i,j)=(a-1,b-1)$, whereas the partition corresponding to $(r_0,s_0)$ would not; which is a contradiction.

If the lines associated to $(r_0,s_0)$ and $(r_1,s_1)$ intersect in the positive quadrant (as illustrated in \autoref{figure: lemma 1 case 1}), say at some (non-integer) point $(u,v)$, then both points $(r_0,s_0)$ and $(r_1,s_1)$ belong to the hyperbola $(r-u)(s-v)=uv$. As one moves along this hyperbola, associated lines rotate around $(u,v)$. Without loss of generality, one may assume that $r_0\le r_1$ and $s_0\ge s_1$. Consider the segment of the hyperbola $(r-u)(s-v)=uv$ lying between the points $(r_0,s_0)$ and $(r_1,s_1)$. Suppose that for some point $(r_t,s_t)$ on this segment the line $x/r_t+y/s_t=1$ passes through an integer point $(a,b)$. If $a<u$, then it follows that the triangular partition corresponding to $(r_1,s_1)$ contains the cell $(i,j)=(a-1,b-1)$, whereas the partition corresponding to $(r_0,s_0)$ does not. Similarly, if $a>u$, then the triangular partition corresponding to $(r_0,s_0)$ contains the cell $(i,j)=(a-1,b-1)$, whereas the partition corresponding to $(r_1,s_1)$ does not. Moreover, $a\neq u$ since the hyperbola $(r-u)(s-v)=uv$ can only intersect a vertical line at one point. Therefore, we get a contradiction in this case as well.
\end{proof}

Let $\tau$ be a triangular partition, and set
\begin{equation*}
\mathcal{R}_\tau:=\mathrm{CL}\big\{(r,s)\in \mathcal{Q}\,\big|\  x/r+y/s=1\ \textrm{cuts off}\  \tau\big\}.
\end{equation*}
This is the closure of the connected component corresponding to $\tau$ by the above lemma. Observe that, for any $(r,s)$ in the interior $\mathcal{R}_\tau^\circ$, the line $x/r+y/s=1$ does not touch any cells. With this terminology, the bijective map of \autoref{lemma_bijection} sends $ \mathcal{R}_\tau^\circ \mapsto\tau$. 

We say that a cell of a triangular partition $\tau$ is \define{removable}  
if it can be removed 
 so that the resulting partition is also triangular. Similarly, we say that a cell of the complement of $\tau$ is \define{addable} if it can be added to $\tau$ so that the resulting partition is also triangular. An argument similar to the proof of \autoref{lemma_bijection}  proves the following:

\begin{lemma}
Let $(i,j)$ be a removable cell of a triangular partition $\tau$. Then there exists $(r,s)\in \mathcal{R}_\tau$ such that $(i,j)$ is the unique cell touched by the line $x/r+y/s=1$. 
\end{lemma} 

\begin{proof}
Consider $(r_0,s_0)$ such that the line $x/r_0+y/s_0=1$ cuts off $\tau\setminus \{(i,j)\}$, and $(r_1,s_1)$ such that the line $x/r_1+y/s_1=1$ cuts off $\tau$. If these lines do not intersect in the positive quadrant, then $r_0<r_1$, $s_0<s_1$, and $(a,b)=(i+1,j+1)$ is the only integer point that lies between the lines. Then for some $0<t<1$, with $r_t=r_0+t(r_1-r_0)$ and $s_t=s_0+t(s_1-s_0)$,  the line $x/r_t+y/s_t=1$  cuts off $\tau$ and touches the cell $(i,j)$, while not touching any other cells.

Suppose the lines $x/r_0+y/s_0=1$ and $x/r_1+y/s_1=1$ intersect at some positive point $(u,v)$. One can choose these lines so that they do not pass through integer points, ensuring that $(u,v)$ is not integral. Let us assume that $a=i+1<u$ (the case $a>u$ is obtained similarly). The triangle bounded by the lines and the vertical axis contains the unique integer point $(a,b)=(i+1,j+1)$, while the triangle bounded by the lines and the horizontal axis does not contain integer points. Therefore, the line connecting $(a,b)$ and $(u,v)$ cuts off $\tau$ and touches the unique cell $(i,j)$.
\end{proof}

The following Lemma can be observed directly, but it is especially natural from the point of view of the moduli space of lines $\mathcal{Q}$:

\begin{lemma}\label{lemma: even cycle}
Consider a line $x/r+y/s=1$. Let $(i_1,j_1),(i_2,j_2),\ldots,(i_k,j_k)$ be all the cells touched by $x/r+y/s=1$, ordered so that $i_1<i_2<\ldots<i_k$. Then the set of triangular partitions $\beta$ such that $(r,s)\in{\mathcal{R}}_\beta$ is the union of the following:
\begin{enumerate}
\item $\tau_{rs},$
\item $\beta:=\tau_{rs}-\{(i_1,j_1),(i_2,j_2),\ldots,(i_k,j_k)\},$
\item $\left\{\tau_{rs}-\{(i_l,j_l),(i_{l+1},j_{l+1}),\ldots,(i_k,j_k)\}\,|\,1<l\le k\right\},$
\item $\left\{\tau_{rs}-\{(i_1,j_1),(i_2,j_2),\ldots,(i_l,j_l)\}\,|\,1\le l< k\right\}.$
\end{enumerate}
\end{lemma} 

Indeed, there are $k$ hyperbolas intersecting at $(r,s)$, and these partitions correspond to the $2k$ connected components of $\overline{\mathcal{Q}}$ near $(r,s)$. Note that only two of the cells $(i_1,j_1),\ldots,(i_k,j_k)$ in $\tau_{rs}$ are removable: $(i_1,j_1)$ and $(i_k,j_k)$. In cases $(3)$ and $(4)$ of the Lemma the only removable cell touched by the line $x/r+y/s=1$ is $(i_l,j_l)$. Similarly, of the cells $(i_1,j_1),\ldots,(i_k,j_k)$ only $(i_1,j_1)$ and $(i_k,j_k)$ are addable for $\beta$, and in the cases $(3)$ and $(4)$ the only addable cells touched by the line $x/r+y/s=1$ are $(i_{l-1},j_{l-1})$ and $(i_{l+1},j_{l+1})$ respectively.

\section{Orders on triangular partitions}\label{section_order}
Two orders are of interest here, both obtained by restriction of classical orders on partitions to triangular partitions.
First, even though the \define{dominance order} is only a partial order on all partitions of a given size\footnote{Starting with size 6.}, its restriction to triangular partitions is a total order which we denoted by $\alpha \preceq \beta$. See \autoref{FigDominanceOrder} for an illustration of the difference between the two contexts. 
\begin{figure}[ht]
\includegraphics[width=.9\textwidth]{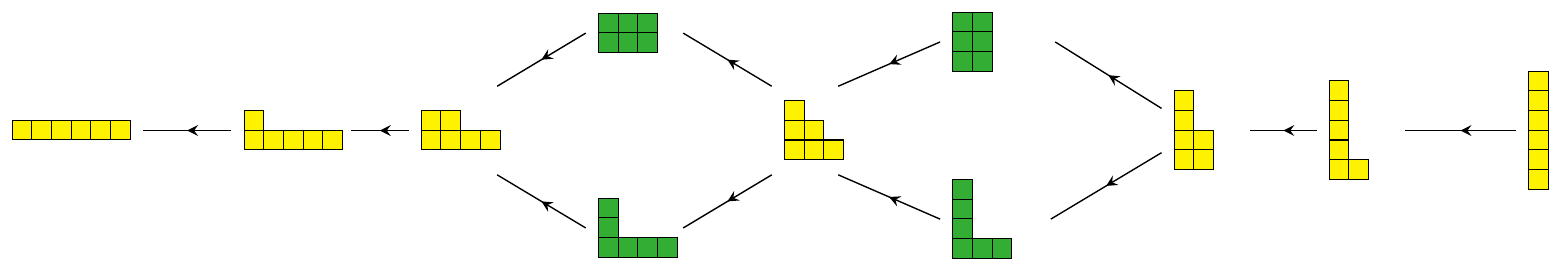}\\
\includegraphics[width=.9\textwidth]{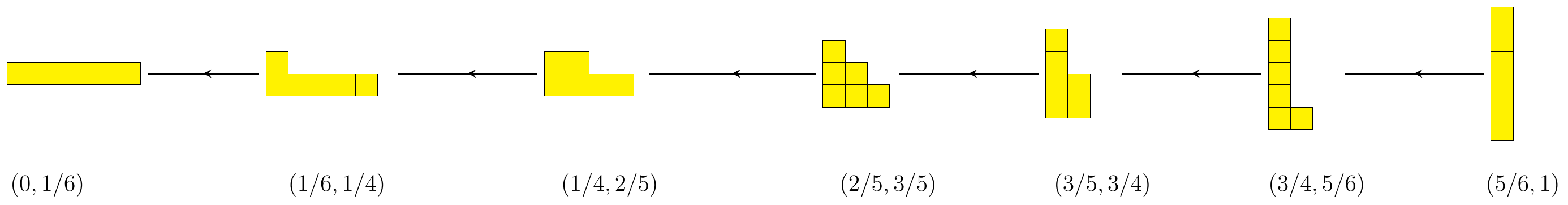}
\caption{The dominance order on all partitions of $6$, and its restriction to triangular ones.}
\label{FigDominanceOrder}
\end{figure}
It is easy to show that
\begin{lemma}\label{lemma: dominance order}
   For two same size triangular partitions $\alpha$ and $\beta$, we have $\alpha \prec \beta$ if and only if $t^+_\alpha\leq t^-_\beta$. Furthermore, $\beta$ covers $\alpha$ in dominance order if and only if $t^+_\alpha=t^-_\beta$.
\end{lemma}

\begin{proof}
According to \autoref{corollary: tripar of size n}, $(t,1-t)$ is not a slope vector of any triangular partition of size $n$ if and only if there exist triangular partitions $\alpha$ and $\beta$ of size $n$, such that $t=t^+_\alpha=t^-_\beta$. Let $x/r+y/s=1$ be the closest to the origin line with the slope vector $(t,1-t)$ such that $|\tau_{rs}|>n$. It follows then that $x/r+y/s=1$ touches more then one cell. Similar to \autoref{lemma: even cycle}, let $(i_1,j_1),(i_2,j_2),\ldots,(i_k,j_k)$ be all the cells touched by $x/r+y/s=1$, ordered so that $i_1<i_2<\ldots<i_k$. Let $l:=|\tau_{rs}|-n$. One gets $l<k$ by the definition of the line $x/r+y/s=1$.  Partition $\tau_{rs}-\{(i_1,j_1),\ldots,(i_l,j_l)\}$ is of size $n$ and can be cut off by a line with a slope vector $(t-\epsilon,1-t+\epsilon)$ for a small enough $\epsilon>0$. Indeed, one should rotate the line $x/r+y/s=1$ counterclockwise around the north-east corner of the cell $(i_{l+1},j_{l+1})$ a little. Therefore, $\alpha=\tau_{rs}-\{(i_1,j_1),\ldots,(i_l,j_l)\}$. By a similar consideration, one can obtain that $\beta=\tau_{rs}-\{(i_{k-l+1},j_{k-l+1}),\ldots,(i_k,j_k)\}$. Finally, one observes that $\alpha \prec \beta$ and the rest of the Lemma follows from \autoref{corollary: tripar of size n}.
\end{proof}

In simple terms, if one lists same size triangular partitions in clockwise order of their slope vectors, they will occur in decreasing dominance order. For size $6$, this is illustrated in \autoref{FigDominanceOrder}.
More properties of the dominance order on triangular partitions (with a different terminology) may be found in~\cite{MR2017036}.

Our second order of interest is the restriction of the containment order $\alpha\subseteq \beta$ to triangular partitions. This gives rise to the \define{Triangular Young poset}, which we denoted by $\YoungTriangle$. We denote the cover relation in $\YoungTriangle$ by $\alpha\youngdelta \beta$.

\begin{lemma}
Let $\alpha$ and $\beta$ be triangular partitions and $\alpha\subset \beta.$ Then one has $\alpha\youngdelta \beta$ if and only if $\alpha$ is obtained from $\beta$ by removing exactly one cell. In particular, $\YoungTriangle$ is ranked by the number of cells in a partition.
\end{lemma}

\begin{proof}
Clearly, if $\alpha$ is obtained from $\beta$ by removing exactly one cell, then $\alpha\youngdelta \beta.$ Suppose now that $\alpha\youngdelta \beta$ and $|\beta|>|\alpha|+1.$ Let $\alpha=\tau_{r_0s_0}$ and $\beta=\tau_{r_1s_1}.$ Similar to Lemma \ref{lemma_bijection}, let us connect the points $(r_0,s_0)$ and $(r_1,s_1)$ of the moduli space $\mathcal{Q}$ by a path $(r_t,s_t)$ as follows: if the corresponding lines do not intersect inside the positive quadrant, then take the straight line $(r_t,s_t)=(r_0+t(r_1-r_0),s_0+t(s_1-s_0));$ if they intersect, then take they rotation around the point of intersection. Observe that in both cases the partition $\tau_{r_ts_t}$ can only increase as $t$ increases. Indeed, even for the rotation, if a cell gets removed, it is not going to be added back later, which contradicts $\tau_{r_0s_0}=\alpha\subset\beta=\tau_{r_1s_1}.$ Since $\alpha\youngdelta \beta$ it follows that all cells from $\beta-\alpha$ have to be added simultaneously, at a certain value $0<t<1,$ which means that the line $x/r_t+y/s_t=1$ touches all the cells of $\beta-\alpha.$ But then Lemma \ref{lemma: even cycle} guarantees that the cells, in fact, can be added one by one. Contradiction. 
\end{proof}

 
\begin{corollary}
Modulo the map $\rho:\mathcal{C}\to \mathcal{T}$ from \autoref{lemma_bijection}, the cover relation of $\YoungTriangle$ corresponds to crossing an hyperbolic wall at a generic point.
\end{corollary}

In other words, the Hasse diagram of $\YoungTriangle$ is planar: for every triangular partition $\tau$ one can draw the vertex corresponding to it at the center of the region $\mathcal{R}_\tau,$ and then connect vertices in the neighboring regions by crossing the shared pieces of the boundary at simple points. Note that the resulting drawing of the Hasse diagram satisfies an additional condition:

\begin{lemma}\label{lemma: planarity}
For any interval $[\alpha,\beta]\subset\YoungTriangle,$ the vertices corresponding to $\alpha$ and $\beta$ in the planar presentation of the Hasse diagram of $[\alpha,\beta]$ obtained by restricting the above construction to $[\alpha,\beta]$ both belong to the boundary of the unbounded region. In particular, the graph obtained from  the Hasse diagram of $[\alpha,\beta]$ by adding an extra edge connecting $\alpha$ to $\beta$ is planar.  
\end{lemma}

\begin{proof}
Indeed, the regions corresponding to the partitions not in the interval $[\alpha,\beta]$ are inside the unbounded region of the Hasse diagram of $[\alpha,\beta],$ and unless $\alpha=\emptyset,$ both $\mathcal{R}_\alpha$ and $\mathcal{R}_\beta$ share boundaries with regions corresponding to partitions outside of $[\alpha,\beta].$ If $\alpha=\emptyset$ then the corresponding vertex is also clearly on the boundary of the unbounded region.
\end{proof}

Lemma \ref{lemma: dominance order} also implies that

\begin{corollary}
In the drawing of the Hasse diagram as above the triangular partitions with the same number of cells appear in the decreasing dominance order, from left to right.
\end{corollary}

There is a subtle difference between the planarity of the Hasse diagram as an abstract graph and its planarity as a Hasse diagram: in a Hasse diagram the edges are required to be drawn going upward from the smaller element to the bigger element, with respect to a chosen direction. See \cite{MR0429672} for a detailed discussion of the planarity of Hasse diagrams and its connection to the condition in Lemma \ref{lemma: planarity}.

\begin{lemma}
The poset $\YoungTriangle$ is a lattice.
\end{lemma}

\begin{proof}
Since $\YoungTriangle$ is bounded from below, it suffices to show that the join operation is well defined. Clearly, for any two triangular partitions $\alpha'$ and $\beta'$ there exist a common upper bound, i.e. a triangular partition $\tau$ such that $\alpha'\subset\tau$ and $\beta'\subset\tau.$ Therefore, we only need to prove that the minimal upper bound is unique. 

Suppose that $A\neq B$ are two minimal upper bounds for $\alpha'$ and $\beta'.$ Consider a saturated chain connecting $\alpha'$ to $A,$ and let $\alpha$ be the maximal element on that chain such that $\alpha\subset B.$ Similarly, consider a saturated chain connecting $\beta'$ to $B,$ and let $\beta$ be the maximal element on that chain such that $\beta\subset A.$ Consider also saturated chains connecting $\alpha$ to $B$ and $\beta$ to $A.$ Note that all four chains connecting $\alpha$ and $\beta$ to $A$ and $B$ cannot intersect each other except at the ends. In particular, they correspond to non-intersecting paths on the Hasse diagram.

Let $\tau$ be a minimal upper bound for $A$ and $B,$ let $T$ be a maximal lower bound for $\alpha$ and $\beta,$ and consider saturated chains connecting $T$ to $\alpha$ and $\beta,$ and $A$ and $B$ to $\tau.$ Note that the corresponding paths on the Hasse diagram do not intersect between each other or with the previously constructed four paths (except at the ends). Finally, Lemma \ref{lemma: planarity} guaranties that if we restrict the drawing of the Hasse diagram to the interval $[T, \tau]\subset \YoungTriangle,$ than the vertices corresponding to $\tau$ and $T$ are going to be on the boundary of the unbounded region, and, therefore, can be connected by an extra path not intersecting any edges of the Hasse diagram of the interval $[T, \tau].$ Thus we obtained a planar presentation of a $K_{3,3}$ graph: every vertex of $\tau, \alpha, \beta$ is connected to every vertex of $T,A,B,$ and all the connecting paths do not intersect each other except at the ends. Contradiction.
\end{proof}

\begin{figure}
\centering
\begin{subfigure}[t]{3.2in}
\centering
 \begin{tikzpicture}
	\node at (0,0) {  \includegraphics[width=.9\textwidth,height=6cm]{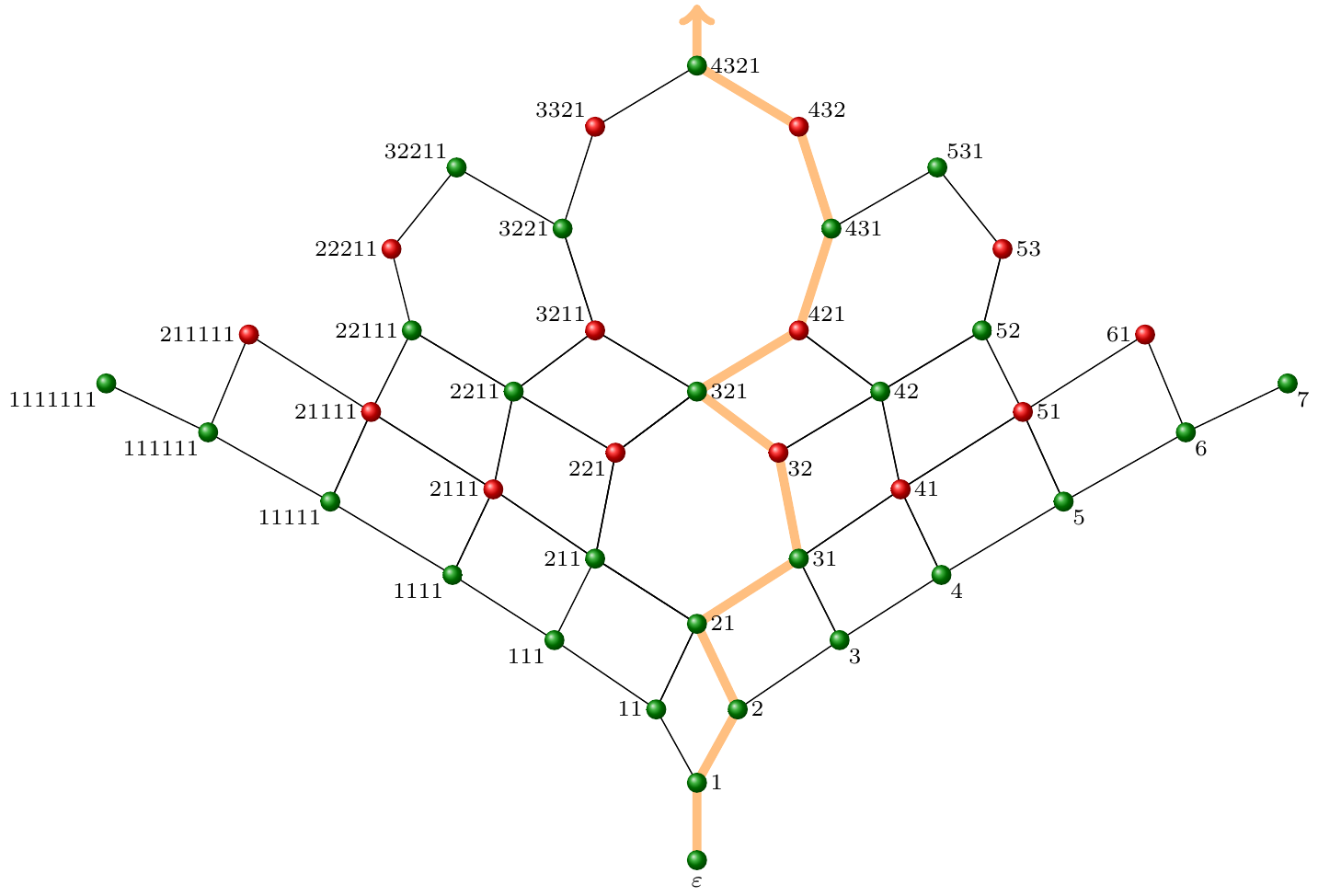}};
\end{tikzpicture}
	\caption{With labeled nodes}\label{Fig_Triangular_Young:1a}
\end{subfigure}
\begin{subfigure}[t]{3.2in}
\centering
   \begin{tikzpicture}
     \node[anchor=south west,inner sep=0] (image) at (-2.1,-2.1) {\includegraphics[width=.9\textwidth,height=6cm]{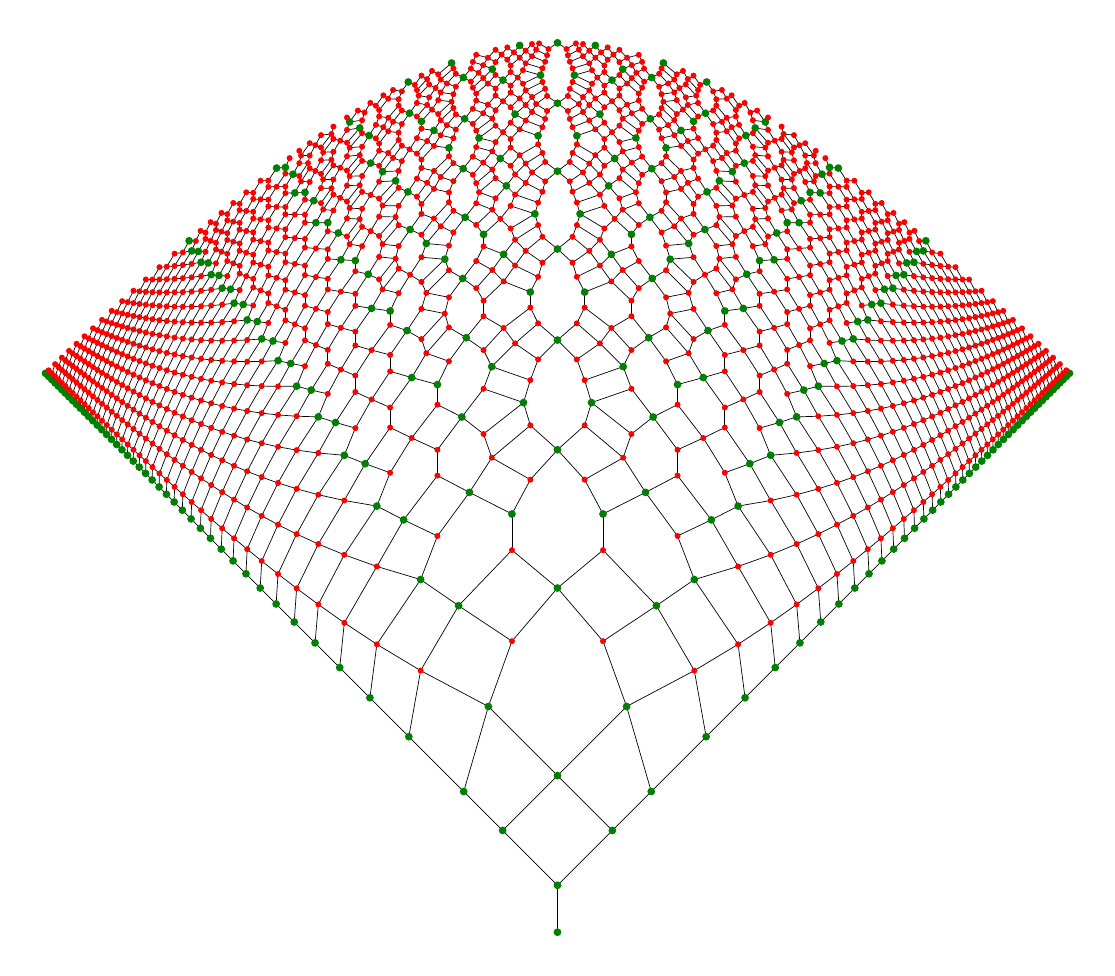}};
    \upnode[60pt]{image}\end{tikzpicture}
 	\caption{Larger portion}\label{Fig_Triangular_Young:1b}
\end{subfigure}
 \caption{Bottom of Triangular Young poset $\YoungTriangle$ (red vertices non-integral).}
 	\label{Fig_Triangular_Young}
\end{figure}

A lower portion of this Hasse diagram is illustrated in \autoref{Fig_Triangular_Young}. Its dual corresponds to \autoref{Fig_Triangular_Young_dual}. 
For obvious reasons, nodes and regions are unlabelled in the larger right-hand side images. The larger image of the poset displays all triangular partitions of $n\leq 45$. Its vertices are equal to $n\,v_\tau$ (up to a logarithmic rescaling and $\pi/4$ rotation), with  $v_\tau$ the standard\footnote{See definition after \autoref{lemma_slope}.} slope vector of $\tau$. The lattice $\YoungTriangle$ is not a distributive, as illustrated by the fact that we have $221\vee (32\wedge 211)\neq (221\vee 32)\wedge (221\vee 211)$.

Let us prove a few more results about the triangular Young poset $\YoungTriangle$ and the structure of the moduli space of lines $\mathcal{Q}.$

\begin{lemma}\label{lemma: no more than two}
No triangular partition can have more than $2$ removable cells. Similarly, no triangular partition can have more than $2$ addable cells.
\end{lemma}

\begin{proof}
Suppose that a triangular partition $\alpha$ has three removable cells $(i_1,j_1)$, $(i_2,j_2)$, and $(i_3,j_3)$, $i_1<i_2<i_3$. It follows that the cell $(i_2,j_2)$ cannot fit under the line touching the cells $(i_1,j_1)$ and $(i_3,j_3)$. Otherwise, in order to remove the cell $(i_2,j_2)$ one would either have to remove the cell $(i_1,j_1)$ or $(i_3,j_3)$. 

Without loss of generality, one may assume that $i_2-i_1\le i_3-i_2$, which implies $j_1-j_2\le j_2-j_3$. There is then a cell 
\begin{equation*}
(k,l):=(i_2+(i_2-i_1),j_2+(j_2-j_1)),
\end{equation*}
since $j_2+(j_2-j_1)\ge j_2+(j_3-j_2)=j_3\ge 0$ and $i_2+(i_2-i_1)>0.$
Consider the lines $\ell_1$ and $\ell_2$ respectively cutting off $\alpha-\{(i_1,j_1)\}$ and $\alpha-\{(i_2,j_2)\}$. The cell $(k,l)$ fits under the line $\ell_1$, but it does not fit under the line $\ell_2$; hence we have a contradiction. The statement about addable cells is proved analogously. \end{proof}
\begin{figure}  
\begin{tiny}
\begin{tikzpicture}[scale=0.6]
\tyng(0cm,0cm,12,11,9,9,6,5,2,1);
\ylw
\tyng(0cm,7cm,1);
\tyng(4cm,5cm,1);
\tyng(10cm,1cm,1);
\yred
\tyng(8cm,3cm,1);

  \draw (0.5,7.5) node {\tiny{$i_1j_1$}};
  \draw (4.5,5.5) node {\tiny{$i_2j_2$}};
  \draw (10.5,1.5) node {\tiny{$i_3j_3$}};
  \draw (8.5,3.5) node {\tiny{$kl$}};
  \draw[red,dotted,thick] (0.5,7.5)--(8.5,3.5);
  \draw[blue] (0,8.25) -- (19,0);
  \draw[blue] (0,8.7) -- (15,0);
  \draw[blue] (12.5,1) node {$l_2$};
  \draw[blue] (14.5,2.5) node {$l_1$};
\end{tikzpicture}
\end{tiny}
\caption{Illustration of \autoref{lemma: no more than two}.}
\label{figure: no more than 2}  
\end{figure}

One case of  \autoref{lemma: no more than two} is illustrated in \autoref{figure: no more than 2}.  Here the line $l_1$ cuts off $\alpha-\{(i_1,j_1)\}$ and the line $l_2$ cuts off $\alpha-\{(i_2,j_2)\}$. Then the box $(k,l)$ fits under $l_1$, but not under $l_2.$

\begin{lemma}
Suppose that a triangular partition $\tau$ corresponds to an unbounded region. Then $\tau$ is either empty, a row partition, or a column partition.
\end{lemma}

\begin{proof}
Indeed, if both cells $(1,0)$ and $(0,1)$ are in $\tau$ then both the leg and the arm of $(0,0)$ in $\tau$ are positive. \autoref{lemma_slope} then implies that there exist $\mu>\lambda>0$ such that if the line $x/r+y/s=1$ cuts off $\tau$ then $\mu>r/s>\lambda$. Also, if $(a-1,b-1)$ is an addable cell, then the hyperbolic wall $(r-a)(s-b)=ab$ is an upper boundary of the corresponding region. But then, the region has to lay inside the curved triangle bounded by the lines $r=\mu s$ and $r=\lambda s$, and the hyperbola $(r-a)(s-b)=ab.$
\end{proof}

\begin{lemma}\label{lemma: no alternating}
Any bounded region is either a triangle or a quadrilateral. Each quadrilateral region has two lower boundaries and two upper boundaries, and they cannot alternate, \ie  the two lower boundaries intersect in a vertex of the region, and the two upper boundaries intersect in a vertex of the region. 
\end{lemma}

\begin{proof}
A bounded region cannot have just two boundaries, since two hyperbolic walls cannot intersect at more than one point (that would correspond to two distinct lines passing through the same two points). In view of \autoref{lemma: no more than two}, the only thing left to check is that a quadrilateral region cannot have sides alternating between lower and upper boundaries. 

Let us follow the boundary of a region in the counterclockwise direction, and suppose that at a certain vertex we have a change in the type of the boundary (from lower to upper, or vice-versa). Let $(r-a)(s-b)=ab$ be the hyperbola that we followed approaching the vertex, and let $(r-c)(s-d)=cd$ be the hyperbola we followed after the vertex. Perforce, we have $c>a$. Indeed, when moving counterclockwise along a lower boundary, $r$ is decreasing and $s$ is increasing, hence the corresponding line is rotating counterclockwise around $(a,b)$. For the next boundary to be upper, one has to hit the north-east corner $(c,d)$ of an addable cell $(c-1,d-1)$, which has to lie to the east of $(a,b)$ since the line is rotating counterclockwise. The case when the boundary changes from upper to lower is similarly dealt with. One concludes that the upper and lower boundaries cannot always alternate as one moves around the region.
\end{proof}

\begin{lemma}\label{lemma: parallel}
Let $\alpha$ be a triangular partition corresponding to a quadrilateral region, which is to say that  it has two removable and two addable cells. Then the line touching the two removable cells is parallel to the line touching the two addable cells\footnote{In \autoref{Fig_Triangular_Young_dual} this is reflected by the fact that lines connecting the top and the bottom corners of quadrilateral regions have slopes $1$ (in  logarithmic scale).}.
\end{lemma}

\begin{wrapfigure}{r}{0.35\textwidth}
\vskip-10pt
\includegraphics[width=0.33\textwidth]{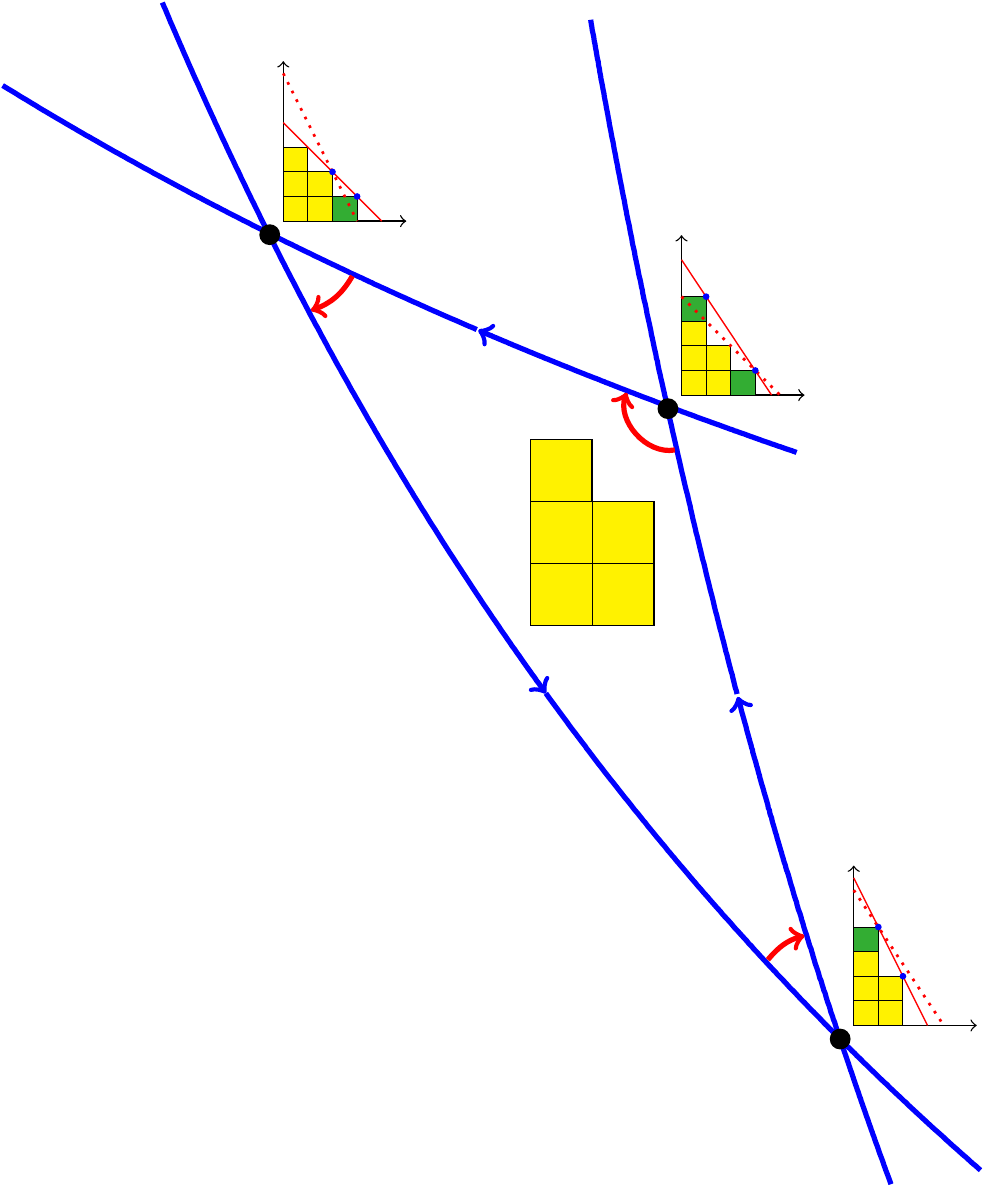}
\vskip-5pt
\wrapcaption{Illustration for \autoref{lemma: no alternating}}
\label{figure: no alternating}
\vskip-30pt
\end{wrapfigure}
As illustrated in \autoref{figure: no alternating}, if one moves counterclockwise along the boundary of the ($3$ sided) region corresponding to $\tau=221$, then the corresponding line successively rotates 
\begin{itemize}
\item[$\bullet$] counterclockwise around the north-east corner of the removable cell $(1,1)$ (left), 
\item[$\bullet$] clockwise around the north-east corner of the addable cell $(2,0)$ (center), 
\item[$\bullet$] and clockwise around the addable cell $(0,3)$ (right). 
\end{itemize}
In a transition from a lower boundary to an upper boundary (left) or vice versa (right), the center of rotation moves east. If the next piece of the boundary is of the same type (center), then  the center of rotation moves west.

\begin{proof}
Let $\ell_1$ be the line touching the two removable cells of $\alpha$, and let $\ell_2$ be the line touching the two addable cells of $\alpha$. \autoref{lemma: no alternating} implies that $\ell_1,\ell_2\in \overline{R_\alpha}$.  Let 
\begin{equation*}
(i_1,j_1),(i_2,j_2),\ldots,(i_k,j_k),\qquad {\rm with}\qquad i_1<i_2<\ldots<i_k,
\end{equation*}
be all the cells touched by $\ell_1$, and let 
\begin{equation*}
(k_1,l_1),(k_2,l_2),\ldots,(k_n,l_n),\qquad {\rm with}\qquad k_1<k_2<\ldots<k_n,
\end{equation*}
be all the cells touched by $\ell_2$. Since $\ell_1,\ell_2\in \overline{R_\alpha}$, according to \autoref{lemma: even cycle} the cells $(i_1,j_1),(i_2,j_2),\ldots,(i_k,j_k)$ are in $\alpha$, and $(i_1,j_1)$ and $(i_k,j_k)$ are the two removable cells of $\alpha$. Similarly, cells $(k_1,l_1),(k_2,l_2),\ldots,(k_n,l_n)$ are not in $\alpha$, and $(k_1,l_1)$ and $(k_n,l_n)$ are the two addable cells of $\alpha.$

Suppose that $\ell_1$ and $\ell_2$ are not parallel.Let $(u,v)$ be their intersection (here $(u,v)$ is not necessarily in the positive quadrant). Note that all the removable cells should fit below $\ell_2$ and none of the addable cells can fit below $\ell_1$.It follows that $(u,v)$ is either to the east of the north-east corners of all addable and removable cells, or it is to the west of all of them. Without loss of generality one can assume that $(u,v)$ is to the east,\ie  $u<a_1+1$ and $u<c_1+1$.It follows then that the line $\ell_1$ is steeper than $\ell_2.$

If there is a cell such that its north-east corner is strictly inside the triangle bounded by the horizontal axis and the lines $\ell_1$ and $\ell_2$, then one gets a contradiction: this cell is strictly below $\ell_2$, but cannot be inside $\alpha$ since its north-east corner is above $\ell_1$.

One needs to consider two cases. Suppose first that $l_1\ge j_1$. Then there is a cell $(k_1+(i_2-i_1),l_1+(j_2-j_1))$ and its north-east corner is inside the required triangle. Indeed, one has:

\begin{equation*}
l_1+(j_2-j_1)\ge j_1+j_2-j_1=j_2\ge 0,
\end{equation*}
and $k_1+(i_2-i_1)>0$, so it is a cell, and its north-east corner is below $\ell_2$ and above $\ell_1$, because we moved from the cell $(k_1,l_1)$  with the north-east corner on $\ell_2$ and above $\ell_1$ in the direction parallel to $\ell_1$, which goes down steeper than $\ell_2.$

Suppose now that $l_1<j_1$.Then there is a cell $(i_1+(k_2-k_1),j_1+(l_2-l_1))$, and it is inside the required triangle. Indeed, one has:

\begin{equation*}
j_1+(l_2-l_1)> l_1+l_2-l_1=l_2\ge 0,
\end{equation*}
and $i_1+(k_2-k_1)>0$,  so it is a cell, and its north-east corner is below $\ell_2$ and above $\ell_1$ because we moved from the cell $(i_1,j_1)$ with the north-east corner on $\ell_1$ and below $\ell_2$ in the direction parallel to $\ell_2$, which goes down but less steep than $\ell_1.$
\end{proof}

\begin{figure}
\includegraphics[scale=.8]{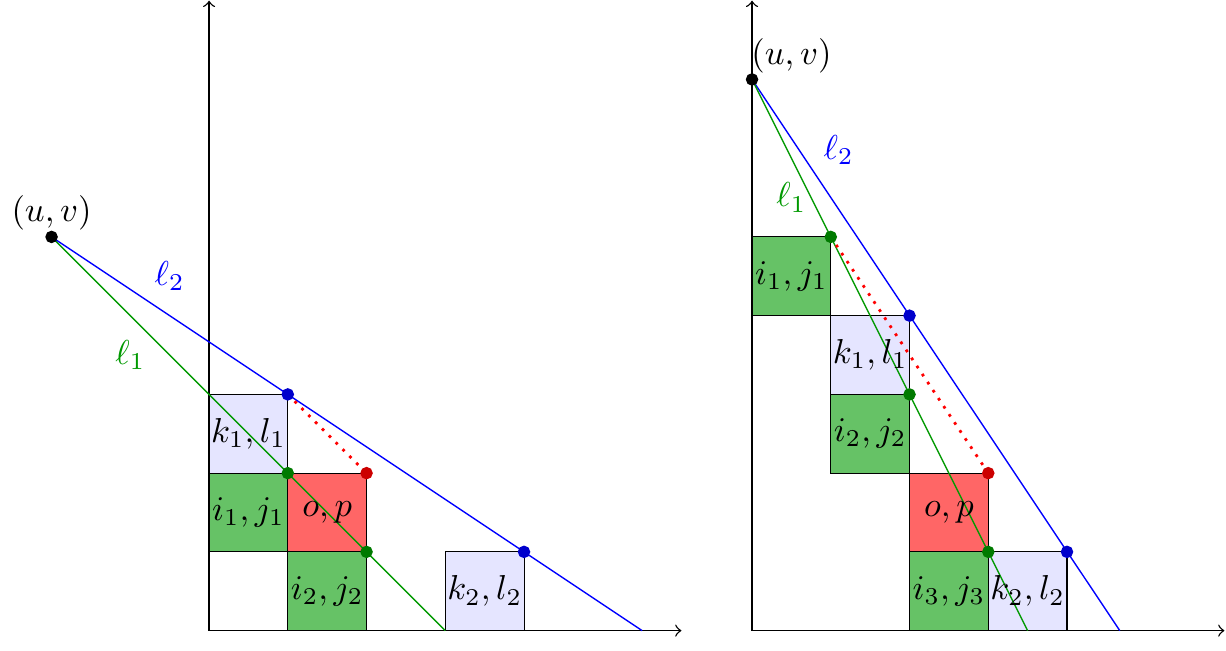}
\caption{Illustration for \autoref{lemma: parallel}}
\label{figure: parallel}    
\end{figure}
Illustrated in \autoref{figure: parallel}, are the cases when $l_1>j_1$ (left) and $l_1<j_1$ (right). Correspondingly, on the left $(o,p)=(k_1+(i_2-i_1),l_1+(j_2-j_1))$ and on the right $(o,p)=(i_1+(k_2-k_1),j_1+(l_2-l_1))$. In both instances the cell $(o,p)$ creates a contradiction as it fits under $\ell_2$ but not under $\ell_1.$

\begin{lemma}\label{lemma: triangle with two lower}
Suppose that a triangular partition $\alpha$ has two removable cells and just one addable cell (in other words it corresponds to a triangular region with two lower boundaries). Then the line touching the two removable cells does not contain any other positive integer points. Equivalently, no other hyperbolic wall passes through the vertex of the region where the two lower boundaries intersect.
\end{lemma}

\begin{proof}
Suppose that two parallel lines $\ell_1$ and $\ell_2$ are such that 
\begin{enumerate}
\item they both contain integer points,
\item there are no integer points (even not necessarily positive) between them,
\item $\ell_1$ contains finitely many positive integer points.
\end{enumerate}
Then $\ell_2$ also contains finitely many positive integer points, and the numbers of positive integer points on $\ell_1$ and on $\ell_2$ differ not more than by one.

Now, let $\ell_1$ be the line touching the two removable cells of $\alpha$, and let $\ell_2$ be the line parallel to $\ell_1$, above it, and satisfying the above conditions. Clearly, $\ell_2\in\overline{R_\alpha}$.The statement above then implies that $\ell_2$ contains at least one positive integer point. But if it contains more than one positive integer point, then by \autoref{lemma: even cycle}, the east-most and the west-most cells touched by $\ell_2$ are both addable. Contradiction. Therefore, $\ell_2$ contains exactly one positive integer point. But then $\ell_1$ cannot contain more than two. 
\end{proof}

\begin{lemma}
Suppose that a triangular partition $\alpha$ has just one removable cell and two addable cells (in other words it corresponds to a triangular region with two upper boundaries). Then the line touching the two addable cells does not contain any other positive integer points. Equivalently, no other hyperbolic wall passes through the vertex of the region where the two upper boundaries intersect).
\end{lemma}

\begin{proof}
Similar to \autoref{lemma: triangle with two lower}.
\end{proof}

\begin{figure}[ht]
\includegraphics[scale=.8]{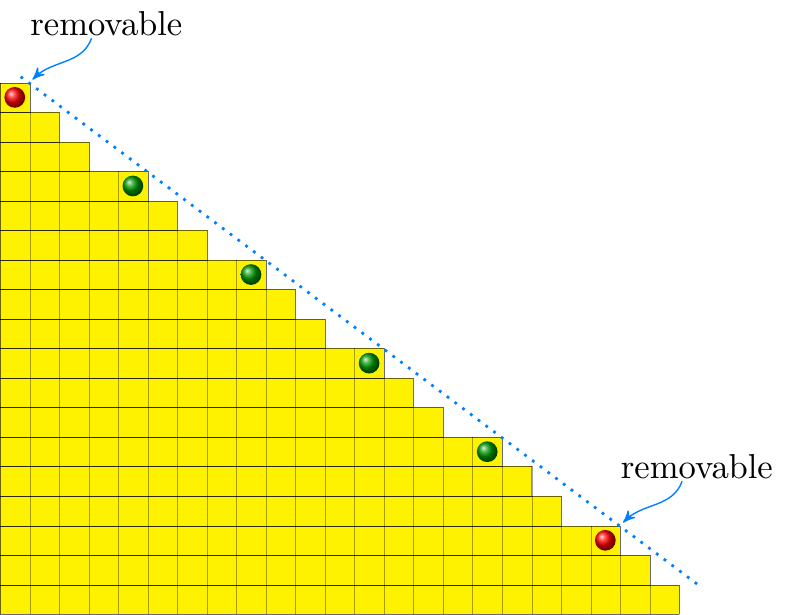}
\caption{The diagonal of $\tau$ is the segment bounded by its removable cells.}
\label{FigTriangle}
\end{figure}
The \define{diagonal} of triangular partition $\tau$, is the set of cells that lie on the segment joining its removable cells. If there is just one such cell, the diagonal is reduced to a single cell.  The diagonal acts as a natural ``boundary'' of $\tau$, and we denote it by $\partial_\tau$. As illustrated in \autoref{FigTriangle}, the cells of $\partial_\tau$ are corners (either in red or green) of $\tau$.  It follows from Lemma \ref{lemma: even cycle} that the partition $\tau^\circ:=\tau\setminus\partial_\tau$, which we call the \define{interior} of $\tau$, is a triangular partition. 
Thus, when $\delta_\tau$ contains $k$ cells, the Hasse diagram of the interval $[\tau^\circ,\tau]$ is a $2k$-sided polygon. In particular, the interval $[\tau_{k-1,k-1},\tau_{kk}]$ is always $2k$-gon. This is why the whole poset is a mosaic of $2k$-gons, as illustrated in \autoref{Fig_Triangular_Young}. 

For a given slope vector $v=(t,1-t)$, the \define{ray} $R_v$ is the set of  triangular partitions having $v$ as an \define{admissible} slope vector, \ie  
    \begin{equation}
       \bleu{R_v} :=\bleu{ \{\tau\in\YoungTriangle \, |\,  t_\tau^- <t < t_\tau^+\}}.
   \end{equation}
This is an infinite chain in $\YoungTriangle$ (see \autoref{Fig_Triangular_Young}). For instance, we have the ray
\begin{align*}
 R_{(1-\epsilon,1+\epsilon)}&= \{\varepsilon,1, 2, 21, 31, 32, 321, 421, 431, 432, 4321,\ldots\}. 
 \end{align*}
See also illustration in \autoref{Fig_Triangular_Young}. As any line of irrational slope contains at most one integral point, it follows that for all $n\in\N$ there is one and only one triangular partition of size $n$ on any ray associated to an irrational slope.

\newpage

\section{Triangular Dyck paths}\label{sec_dyck}
\begin{wrapfigure}{r}{0.37\textwidth}
\vskip-13pt
\includegraphics{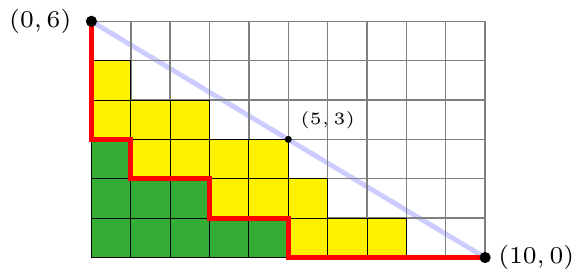}
\vskip-5pt
\wrapcaption{The $\tau_{(10,6)}$-Dyck path $531000$}
\vskip-10pt
\end{wrapfigure}
For a given triangular partition $\tau$, we consider the set $\bleu{ \Dyck_{\tau}}:=\bleu{\{\alpha\,|\, \alpha\subseteq \tau\}}$ of \define{$\tau$-Dyck paths}.
Observe that  conjugation gives a bijection between  $\Dyck_{\tau}$ and $\Dyck_{\tau'}$. 
We further write $\Dyck_{(r,s)}$, when $\tau=\tau_{rs}$, and say that its elements are 
\define{$(r\times s)$-Dyck paths}. Observe that  $\Dyck_{(r,n)}=\Dyck_{(kn,n)}$, for all $kn\leq r\leq kn+1$ and $k\in\N$, since the corresponding triangular partitions coincide. It is often convenient to consider that partitions $\Dyck_{\tau}$ are padded with zero parts to make them of length $n=l(\tau)+1$. 
As we have already mentioned, classical Dyck paths correspond to the case $\tau=\tau_{n,n}$, with $n\in\N$. For instance, we have the equivalent descriptions
\begin{align*}
   \Dyck_{(3,3)} 
                      &= \{210,110,200,100,000\},\\
                      &=\begin{Bmatrix} 
                      		\trois{2}{1}, 
				&\trois{1}{1}, 
				&\trois{2}{0}, 
				&\trois{1}{0}, 				
				&\trois{0}{0}
			\end{Bmatrix},\\
			&=\begin{Bmatrix} 
\begin{tikzpicture}[scale=.4]
\Yfillcolour{yellow!20}
\draw[blue,opacity=.2] (0,3)-- (3,0);
\tyng(0cm,0cm,2,1);
\draw [red,line width=.5mm] (0,3) -- (0,2) --(1,2) -- (1,1) -- (2,1) --  (2,0)-- (3,0);
\end{tikzpicture}, 
&
\begin{tikzpicture}[scale=.4]
  \Yfillcolour{yellow!20}
\draw[blue,opacity=.2] (0,3)-- (3,0);
\tyng(0cm,0cm,2,1);
\draw [red,line width=.5mm] (0,3) -- (0,2) --(1,2) -- (1,0)-- (3,0);
\end{tikzpicture}, 
&
\begin{tikzpicture}[scale=.4]
\Yfillcolour{yellow!20}
\draw[blue,opacity=.2] (0,3)-- (3,0);
\tyng(0cm,0cm,2,1);
\draw [red,line width=.5mm] (0,3) -- (0,1) --(2,1) -- (2,0)-- (3,0);
\end{tikzpicture}, 
&, 
\begin{tikzpicture}[scale=.4]
\Yfillcolour{yellow!20}
\draw[blue,opacity=.2] (0,3)-- (3,0);
\tyng(0cm,0cm,2,1);
\draw [red,line width=.5mm] (0,3) -- (0,1) --(1,1) -- (1,0)-- (3,0);
\end{tikzpicture}, 
&
\begin{tikzpicture}[scale=.4]
\Yfillcolour{yellow!20}
\draw[blue,opacity=.2] (0,3)-- (3,0);
\tyng(0cm,0cm,2,1);
\draw [red,line width=.5mm] (0,3) -- (0,0)-- (3,0);
\end{tikzpicture}
\end{Bmatrix}.
\end{align*}
As is very well known, these are counted by the Catalan numbers $\Cat_n=\frac{1}{n+1}\binom{2n}{n}$. In general, we set \define{$\Cat_\tau:=\#\Dyck_\tau$}.
For integral partitions $\tau_{mn}$, the greatest common divisor $d=\gcd(m,n)$ plays an interesting role in our story. We have $m=ad$ and $n=bd$, for coprime integers $a$ and $b$. The cells of $\tau_{mn}$ lying on its diagonal $\partial(\tau_{mn})$ are of the form $(ak,bk)$, with $0< k< d$.

\begin{wrapfigure}{r}{0.20\textwidth}
\vskip-10pt
\includegraphics[width=0.18\textwidth]{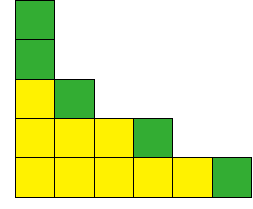}
\wrapcaption{$(\alpha+1^5)/\alpha$}
\label{Figure_Skew_Diag}
\vskip-10pt
\end{wrapfigure}
In preparation for upcoming notions, it will be interesting to consider the skew partition $(\alpha+1^n)/\alpha$, in which $\alpha+1^n$ stands for the partition obtained by adding one to each of the parts of $\alpha$ including its zero parts. The above skew partition is readily seen to consist of ``independant'' columns, as illustrated in \autoref{Figure_Skew_Diag} with $\alpha=53100$.
The sequence of sizes of these columns is precisely $\mu(\alpha)=21011$. The choice of $n$ will depend on the context, and it will be larger than the number of parts of $\alpha$.
For a fixed $\tau=\tau_1\tau_2\cdots \tau_n$, the \define{$\tau$-area} (or simply \define{area}) of a $\tau$-Dyck path $\alpha$ is the number of cells lying  in the skew shape $\tau/\alpha$. In terms of the \define{$\tau$-area sequence}  $(a_i)_{1\leq i\leq n}$ of $\alpha$, in which  $a_i:=\tau_i-\alpha_i$, we clearly have  $\area_{\tau}(\alpha):=\sum_{i=1}^n a_i$.
\begin{definition}\label{define_j_area}
 For any $J$ subset of $\{i\,|\, 1\leq i\leq n\}$, we consider the $J$-area: 
   $$\bleu{\area_{\tau}^J(\alpha)}:=\bleu{\sum_{i\in J} a_i}.$$
 \end{definition}

\begin{figure}[ht]
\includegraphics[scale=.9]{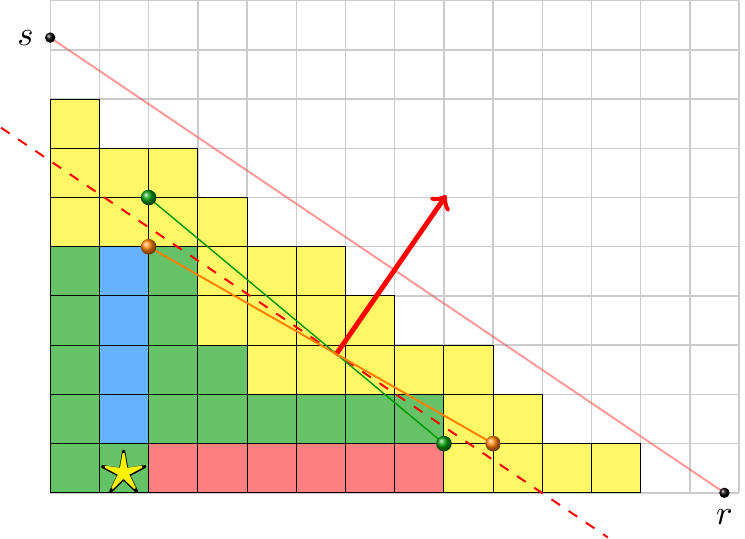}
\caption{A sim cell of $\alpha=(8,8,4,3,3)$ relative to $\tau=(12,10,9,7,6,4,3,1)$.}\label{Fig_Dinv_Set}
\end{figure}
\subsection{Similar cells (Diagonal inversions)} 
Let $\tau$ be a given triangular partition. 
With the notations of \autoref{def_slope_t}, for any $\alpha\subseteq \tau$ consider the set of cells $c$ of $\alpha$ that have hooks in $\alpha$ that are ``similar'' to $\tau$:
 $$\bleu{\Sim_\tau(\alpha)}:=\bleu{\{ c\in \alpha\ |\ t'(c,\alpha)\leq t_\tau<t''(c,\alpha)) \}},$$ 
where $t_\tau:=({t^{-}_\tau+t^{+}_\tau})/{2}$.
The cardinality of this set, denoted by \define{$\simd_\tau(\mu)$}, is the \define{similarity index}\footnote{This just another name for the ``dinv'' statistic, that we choose to stress its natural geometrical meaning.} of $\mu$ with respect to $\tau$. Expressed in words, the \define{similarity index} of $\mu$ with respect to $\tau$ counts the number of cells of $\mu$ having ``hook triangle slope vectors''  compatible with the (average) slope vector of $\tau$ (represented by $\overline{v}$ in \autoref{fig_hook_slopes}). Since at corner cells, we have $\ell=0$ and $a=0$, we get
$\ell/(a+\ell+1)=0$ and $(\ell+1)/(a+\ell+1)=1$. Hence, any corner cell of $\alpha$ is always a $\tau$-sim-cells, irrespective of $\tau$.
The sim-cells of the partitions contained in $\tau=32$ are marked by stars in \autoref{FigDinv}.

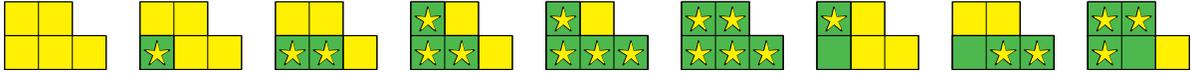
\begin{figure}[ht]
\begin{tikzpicture}[scale=.45]
\ylw
\foreach \x in {0,4,8,12,16,20,24,28,32} {\tyng(\x cm,0cm,3,2);}
\Yfillcolour{green!70}
\tyng(4cm,0cm,1); \Star{4}{0}
\tyng(8cm,0cm,2); \Star{8}{0}\Star{9}{0}
\tyng(12cm,0cm,2,1);  \Star{12}{0}
                                  \Star{12}{1}
                                  \Star{13}{0}
\tyng(16cm,0cm,3,1); \Star{16}{0}\Star{16}{1}\Star{17}{0}\Star{18}{0}
\tyng(20cm,0cm,3,2); \Star{20}{0}\Star{20}{1}\Star{21}{0}\Star{21}{1}\Star{22}{0}
\tyng(24cm,0cm,1,1); \Star{24}{1}
\tyng(28cm,0cm,3); \Star{29}{0}\Star{30}{0}
\tyng(32cm,0cm,2,2); \Star{32}{0}\Star{32}{1}\Star{33}{1}
\end{tikzpicture}
	\caption{Sim-cells for subpartitions of $32$. The first 6 subpartitions are similar to $32$.}
	\label{FigDinv}
\end{figure}

The following is an example of a sequence of similar triangular partitions:
\begin{center}
\ylw
 \Yboxdim{.2cm}
\begin{align*}
&\varepsilon,\yng(1), \yng(1,1),\yng(2,1), \yng(2,1,1),\yng(2,2,1),\yng(3,2,1), \yng(3,2,1,1),\yng(3,2,2,1),\yng(3,3,2,1),\yng(4,3,2,1),
\yng(4,3,2,1,1),\yng(4,3,2,2,1),\yng(4,3,3,2,1),\yng(4,4,3,2,1),\yng(5,4,3,2,1).
\end{align*}
\end{center}
They all share the common slope vector $(1/2+\epsilon,1/2-\epsilon)$.

\begin{lemma}[with B. Dequ\^{e}ne]\label{Lemma_unique_complementary_area}
If $\tau$ is a triangular partition, then the subpartitions $\alpha\subseteq \tau$ such that $\area_\tau(\alpha)+\simd_\tau(\alpha)=|\tau|$ are exactly those that lie on the ray corresponding to the slope vector $(t_\tau+\epsilon,1-t_\tau-\epsilon)$. 
\end{lemma}

\section{Counting \texorpdfstring{$\tau$}{t}-Dyck paths}
The enumeration of triangular Dyck paths has an interesting (ongoing) history, which has up to now been restricted to the integral case. Even if the simple counting in the case of any integer pairs $(m,n)$ had already been worked out in the 1950s (see~\cite{MR61567}), it is only rather recently that the overall enumerative combinatorics community has become aware of that fact. For a long time, only  the ``coprime case'' was deemed really understood, recently going under the name of rational Catalan combinatorics. 
These include the classical \define{Fuss-Catalan} numbers when $m=kn$ (or equivalently when $m=kn+1$). A direct extension of the classical ``cycling'' argument shows that, for $m$ and $n$ coprime integers, the number $(m,n)$-Dyck paths is given by the formula
\begin{equation}
\bleu{\Cat_{mn}}=\bleu{\frac{1}{m+n}\binom{m+n}{n}}.
\end{equation}
A formula for the non-coprime case of $\tau_{mn}$ is described in the next subsection.

\subsection{Bizley formula}
In the general ``integral'' situation, the enumeration formula of $(m\times n)$-Dyck paths takes the form of a sum of terms indexed by partitions of the greatest common divisor of $m$ and $n$. This is what makes it harder to ``guess'' a formula\footnote{Most guessing approaches rely (directly or indirectly) on the fact the numbers considered have nice factorization in small prime numbers.}, since the numbers obtained do not factor nicely in general, even if they are effectively sums of nicely factorized numbers. The \define{Grossman-Bizley} formula (see~\cite{MR61567}) is:
  \begin{equation}\label{formule_bizley}
      \bleu{\Cat_{mn}}:=\bleu{\sum_{\alpha\vdash d} \frac{1}{z_\alpha}\prod_{k\ \hbox{\tiny \rm part\ of\ } \alpha} \frac{1}{a+b} \binom{k(a+b)}{ka}},
   \end{equation}
where $(m,n)=(ad,bd)$ with $a$ and $b$ coprime, so that $d=\gcd(m,n)$.   It is worth recalling that $d!/z_\alpha$ is the number of size $d$ permutations of cycle type $\alpha$, with
 		\begin{displaymath} \bleu{z_\alpha}:=\bleu{\prod_{i} i^{c_i}\, c_i!},\end{displaymath}
where $\alpha$ has $c_i$ parts of size $i$.
Specific examples of \autoref{formule_bizley} are:
\begin{equation}\label{specific_Bizley}
\begin{split}
  \bleu{\Cat_{2a,2b}}&= \bleu{\textstyle \frac{1}{2} \left(\frac{1}{a+b} \binom{a+b}{a}\right)^2+\frac{1}{2}\left(\frac{1}{a+b} \binom{2a+2b}{2a}\right)},\\[4pt]
  \bleu{\Cat_{3a,3b}}&=\bleu{\textstyle\frac{1}{6}\left(\frac {1}{ a+b} \binom{a+b}{a}\right)^3
  +\frac{1}{2}\left(\frac {1}{a+b} \binom{a+b}{a}\right)\left(\frac{1}{a+b} \binom{2a+2b}{2a}\right) +\frac{1}{3}\left(\frac {1}{a+b} \binom{3a+3b}{3a}\right).}
\end{split}
\end{equation}
Observe that, for fixed coprime numbers $a$ and $b$, all the formulas for $(m,n)=(ad,bd)$, with $0\leq d$, may be presented in the form of the generating function:
\begin{equation}\label{bizley_gen}
   \bleu{ \sum_{d=1}^\infty \Cat_{ad,bd}\ z^d} =\bleu{ \exp\!\Big(\sum_{k\geq 1}{\textstyle \frac{1}{a+b} \binom{k(a+b)}{ka}}\, z^k/k\Big)}.
\end{equation} 
As it happens, this is a specialization of a more general formula (see \autoref{bizley_gen_sym}). 

\subsection{Explicit number of Dyck paths for all triangular partitions}
For all triangular partitions $\tau$ of size at most $9$, the number $\Cat_\tau$ of $\tau$-Dyck paths may be found in the \autoref{TableDyck}.  In the next subsection we will see how to calculate these numbers recursively.
\begin{table}[ht]
\def\arraystretch{1.25}
 \begin{tiny}
\begin{tabular}{|c|c|c|c|c|c|c|c|c|c|c|c|c|c|}
\cline{1-1}
$1$\\ \cline{1-1}
\rowcolor{yellow} $ 2$ \\ \cline{1-2}
$1^2$ & $2$ \\ \cline{1-2}
\rowcolor{yellow} $3$ & $3$ \\
\cline{1-3}
$1^3$ & $21$ & $3$ \\ \cline{1-3}
\rowcolor{yellow} $4$ & $5$ & $4$ \\
\cline{1-4}
$1^4$ & $21^2$ & $31$ & $4$ \\
\cline{1-4}
\rowcolor{yellow}  $5$ & $7$ & $7$ & $5$ \\
\cline{1-6}
$1^5$ & $21^3$ & $2^21$ & $32$ & $41$ & $5$ \\
\cline{1-6}
\rowcolor{yellow} $6$ & $9$ & $9$ & $9$ & $9$ &  $6$ \\
\cline{1-7}
$1^6$ & $21^4$ & $2^21^2$ & $321$ & $42$ & $51$ & $6$ \\
\cline{1-7}
\rowcolor{yellow} $7$ & $11$ &   $12$ &  $14$ &   $12$ & $11$ &  $7$ \\
\cline{1-8}
$1^7$ & $21^5$ & $2^21^3$ & $3211$ & $421$ & $52$ & $61$ & $7$ \\
\cline{1-8}
\rowcolor{yellow}  $8$ & $13$ &  $15$ & $19$ & $19$ &   $15$ & $13$ &  $8$ \\
\cline{1-10}
$1^8$ & $21^6$ & $2^21^4$ & $2^31^2$ & $32^21$ & $431$ & $53$ & $62$ & $71$ & $8$ \\
\cline{1-10}
\rowcolor{yellow} $9$ & $15$ & $18$ & $18$ &   $23$ &   $23$ & $18$ & $18$ & $15$ &  $9$ \\
\cline{1-12}
$1^9$ & $21^7$ & $2^21^5$ & $2^31^3$ & $32^21^2$ & $3^221$ & $432$ & $531$ & $63$ & $72$ & $81$ & $9$ \\
\cline{1-12}
\rowcolor{yellow} $10$ & $17$ & $21$ &  $2^2$ &   $30$ & $28$ & $28$ &  $30$ &  $2^2$ & $21$ & $17$ & $10$ \\
\cline{1-12}
%
\end{tabular}
\end{tiny}
\medskip
\caption{Number of $\tau$-Dyck paths}\label{TableDyck}
\end{table}

\subsection{General recursive formula for triangular partitions}
For any partition triangular $\tau$, the $q$-area enumerator of $\tau$-Dyck paths, is
   \begin{equation}
      \bleu{\Cat_{\tau}(q)}:=\bleu{\sum_{\alpha\subseteq \tau} q^{\area_{\tau}(\alpha)}}.
   \end{equation}
Since conjugation is an area-preserving bijection between the set of $\tau$-Dyck paths and the set of $\tau'$-Dyck paths, we clearly have $ \Cat_{\tau}(q)=\Cat_{\tau'}(q)$. In preparation for the upcoming proposition, let us consider the following notions. The \define{bounding word $w_\mu$}\index{partition!bounding word}, of a partition $\mu$, encodes the simplest southeast lattice path such that the cells of the diagram of $\mu$  are those that sit below $w_\mu$.  Thus $w_\mu:=0^{r_k}1\cdots 0^{r_2}1\cdots 0^{r_{1}}1$,
  where $r_k=\mu_k$, and $r_i=\mu_{i}-\mu_{i+1}$, for $1\leq i< k$. For two partitions $\alpha$ and $\beta$, we denote by $\alpha\odot \beta$ the partition whose bounding word/path is the concatenation of words: $w_\alpha\cdot w_\beta$.  The empty partition acts as the identity for this associative product. 
Three-way decompositions of the form $\nu=\alpha\odot \gamma \odot \beta$, with $\gamma$ indicating a one cell partition, will be of special interest. Observe that the middle one cell partition ``$\gamma$'' necessarily corresponds to a corner of $\nu$. For example,
\begin{figure}[ht]
\includegraphics{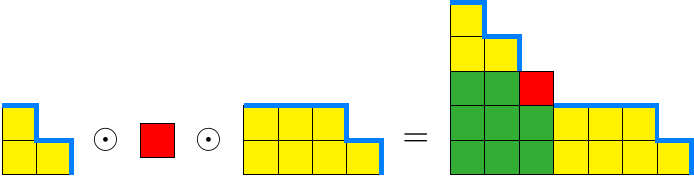}
\caption{A decomposition $\alpha\odot \gamma \odot \beta$.}\label{FigThreeWay}
\end{figure} 
Let us denote by $\Delta(\tau)$ the set of pairs $(\alpha,\beta)$ corresponding to three-way decompositions of $\tau$, of the form $\tau=\alpha\odot \gamma\odot \beta$, with the corner cell $\gamma$ sitting on the diagonal of $\tau$. In formula,
\begin{equation}
	\bleu{\Delta(\tau)} :=\bleu{\{ (\alpha,\beta)\ |\ \tau=\alpha\odot \gamma\odot \beta,\quad {\rm and}\quad 
		\gamma\in\partial_\tau\}}.
 \end{equation}
Recall also that $\tau^\circ$ stands for the interior of $\tau$, which is obtained by removing from $\tau$ all the cells on its diagonal. 
Then, the number $\Cat_\tau$ of $\tau$-Dyck paths may efficiently be calculated with the following recursive formula. 
\begin{figure}[ht]
\includegraphics{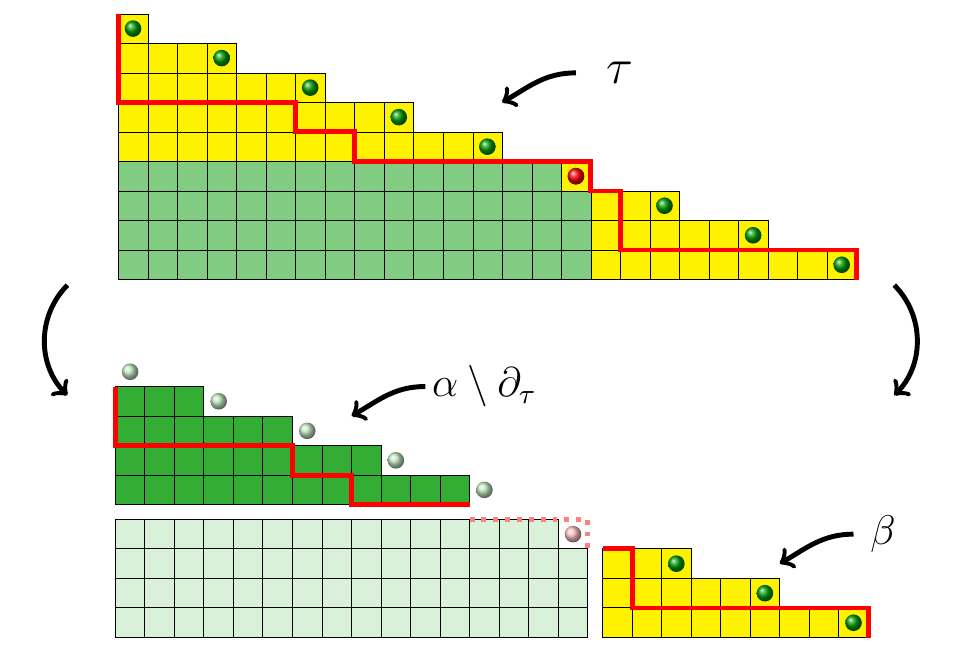}
	\caption{First return  cell \protect\includegraphics{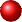}, with \protect\includegraphics{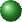} marking other diagonal cells.}
	\label{FigRecurrence}
\vskip-10pt
\end{figure} 

\begin{proposition}\label{proposition_Delta_recurrence}
Denoting by $\partial=\partial_\tau$ the diagonal of a given triangular partition $\tau$, then for the $q$-area enumerator of $\tau$-Dyck, we have the recurrence
\begin{equation}\label{General_q_Catalan_recurrence}
    \bleu{\Cat_\tau(q)}=\bleu{q^{|\partial|}\,\Cat_{\tau^\circ}(q)+\sum_{(\alpha,\beta) \in \Delta(\tau)} q^{|\alpha\cap \partial|}\,\Cat_{\alpha\setminus \partial}(q)\,\Cat_\beta(q)},
\end{equation}
with initial condition $\Cat_\varepsilon(q)=1$.
In particular, setting $q=1$, we have
\begin{equation}\label{General_Catalan_recurrence}
    \bleu{\Cat_\tau}=\bleu{\Cat_{\tau^\circ}+\sum_{(\alpha,\beta) \in \Delta(\tau)} \Cat_{\alpha\setminus \partial} \, \Cat_\beta}.
\end{equation}
\end{proposition}
This is a direct generalization of the well-known classical recurrence for Dyck path. Its proof corresponds to a suitably adapted ``first return to diagonal'' argument typically used in the proof of the classical case (see \autoref{FigRecurrence}). It is noteworthy that all possible $\alpha\setminus \partial$ and $\beta$ that occur in the right-hand side of \autoref{General_q_Catalan_recurrence} are triangular, as they are respectfully factors of $\tau^\circ$ and $\tau$. Hence, for a given partition $\tau$, the set of partitions that will arise in the recurrence can only be factors (for the $\odot$-product) of partitions obtained by successive removal of diagonals. Hence, they all have a slope in common with that of $\tau$.

The $q$-area enumerators for all triangular partitions of size at most $6$ are as follows (avoiding repetitions for conjugate partitions):
$$\begin{array}{ll llrcl}
 \Cat_{1} =& q + 1,& \Cat_{2} =& q^{2} + q + 1,\\
 \Cat_{21} =& q^{3} + q^{2} + 2 q + 1, & \Cat_{3} =& q^{3} + q^{2} + q + 1,\\
\Cat_{31} =& q^{4} + q^{3} + 2 q^{2} + 2 q + 1, &\Cat_{4} =& q^{4} + q^{3} + q^{2} + q + 1,\\
\Cat_{32} =& q^{5} + q^{4} + 2 q^{3} + 2 q^{2} + 2 q + 1,& \Cat_{41} =& q^{5} + q^{4} + 2 q^{3} + 2 q^{2} + 2 q + 1,\\
\Cat_{5} =& q^{5} + q^{4} + q^{3} + q^{2} + q + 1, &\Cat_{321} =& q^{6} + q^{5} + 2 q^{4} + 3 q^{3} + 3 q^{2} + 3 q + 1,\\
\Cat_{42} =& q^{6} + q^{5} + 2 q^{4} + 2 q^{3} + 3 q^{2} + 2 q + 1,&\Cat_{51} =& q^{6} + q^{5} + 2 q^{4} + 2 q^{3} + 2 q^{2} + 2 q + 1,\\
\Cat_{6} =& q^{6} + q^{5} + q^{4} + q^{3} + q^{2} + q + 1.
\end{array}$$

\subsection[Counting by area and Dinv]{Counting by Area and Sim}
As in \cite{2102.07931}, we may consider the enumeration of $\tau$-Dyck paths with respect to two statistics: ``area'' and ``sim'', in the triangular case\footnote{Our ``sim'' is the ``dinv'' of {\sl loc. cit.}}. The resulting $(q,t)$-polynomials:
   \begin{equation}\label{cat_tau_combinatoire}
      \bleu{\Cat_{\tau}(q,t)}:=\bleu{\sum_{\alpha\subseteq \tau} q^{\area_{\tau}(\alpha)}t^{\dinv_{\tau}(\alpha)}},
   \end{equation}
plays a central role in a wide range of subjects. Once again, applying conjugation to partitions $\alpha$ contained in $\tau$, it is easy to check that $\Cat_{\tau'}(q,t)=\Cat_{\tau}(q,t)$. It follows from \autoref{Lemma_unique_complementary_area} that, for any triangular partition $\tau$ of size $n$, we have
\begin{align}
   \bleu{\Cat_{\tau}(q,t)}& = \bleu{(q^n+q^{n-1}t+\ldots + q t^{n-1}+t^n)+\ldots },\nonumber\\
       &=\bleu{s_n(q,t)+\ldots} \label{observation_Dn}
\end{align}
where the remaining terms are of degree (strictly) less than $n$. In the particular case $\delta_k=(k,k-1,\ldots,2,1,0)$, one may further see that
  \begin{equation}
   \bleu{\Cat_{\delta_k}(q,t)}=\bleu{s_{\binom{k+1}{2}}(q,t) +\Big(\sum_{j=\binom{k}{2}}^{\binom{k+1}{2}-2} s_{j,1}(q,t)\Big) + \ldots },
 \end{equation}
 with the missing terms only involving Schur functions indexed by partitions having a second part larger or equal to $2$. Taking into account  the symmetry $\Cat_{\tau}=\Cat_{\tau'}$ to avoid unnecessary repetitions, the values of $\Cat_{\tau}$ for (all) triangular partitions of size at most $8$ are as given in \autoref{table_D_tau}. The entries are expressed in terms of Schur functions $s_\mu=s_\mu(q,t)$, so that the Schur positivity (see next section) of is made apparent.  
 \begin{table}[ht]
\begin{tabular}{lllllllllllllllllll}
$\left(0, \boldsymbol{1}\right)$ \\
$\left(1, s_{1}\right)$ \\
$\left(2, s_{2}\right)$ \\
$\left(21, s_{11} + s_{3}\right)$ & $\left(3, s_{3}\right)$ \\
$\left(31, s_{21} + s_{4}\right)$ & $\left(4, s_{4}\right)$ \\
$\left(32, s_{31} + s_{5}\right)$ & $\left(41, s_{31} + s_{5}\right)$ & $\left(5, s_{5}\right)$ \\
$\left(321, s_{31} + s_{41} + s_{6}\right)$ & $\left(42, s_{22} + s_{41} + s_{6}\right)$ & $\left(51, s_{41} + s_{6}\right)$ & $\left(6, s_{6}\right)$ \\
$\left(421, s_{32} + s_{41} + s_{51} + s_{7}\right)$ & $\left(52, s_{32} + s_{51} + s_{7}\right)$ & $\left(61, s_{51} + s_{7}\right)$ & $\left(7, s_{7}\right)$\\
$\left(431, s_{42} + s_{51} + s_{61} + s_{8}\right)$ & $\left(53, s_{42} + s_{61} + s_{8}\right)$ & $\left(62, s_{42} + s_{61} + s_{8}\right)$ & $\left(71, s_{61} + s_{8}\right)$ & $\left(8, s_{8}\right)$
\end{tabular}
\caption{Table of values $(\tau,\Cat_{\tau})$.}
\label{table_D_tau}
\end{table}

A constant term formula for  $\Cat_{\tau}(q,t)$ is given in~\cite[Prop. 7.2.1]{2102.07931}. It follows that $\Cat_{\tau}(q,t)$ is symmetric in $q$ and $t$, even though this is not evident in \autoref{cat_tau_combinatoire}. 
Setting  $t=1/q$, we often get nice product formulas. For instance,
In the case of $\tau_{ab}$, with $a$ and $b$ coprime integers, we have 
 \begin{equation}
    \bleu{q^{|\tau_{ab}|} \Cat_{\tau_{ab}}(q,1/q)}=\bleu{\frac{1}{[a+b]_q} \qbinom{a+b}{a}}.
 \end{equation} 
 For any triangular partition $\tau=(k,j)$, with two parts, we further have
 \begin{equation}
     \bleu{q^{|\tau|} \Cat_{\tau}(q,1/q)} = \bleu{\frac{[2k-j+2]_q [3(j+1)]_q}{[3]_q[2]_q}}.
  \end{equation}


\section{Generic Schur function expansion}
The polynomials $\Cat_{\tau}(q,t)$ are not only symmetric, but extensive explicit calculations suggest that they always expand positively in the Schur polynomials basis. The following conjecture of~\cite[Conj. 7.1.1]{2102.07931} is supported by extensive calculations (for all triangular partitions of $n\leq 28$); and in some instances proofs and/or justifications via representation theory (see~\cite{1909.03531,MR3811519}).
\begin{conjecture}\label{Conjecture_Schur_positive}
For any triangular partition $\tau$ of size $n$, the polynomial $\Cat_{\tau}(q,t)$ affords a  positive Schur-expansion 
   \begin{equation}
      \bleu{\Cat_{\tau}(q,t)}:=\bleu{\sum_{\lambda} c_\lambda^{\tau}\, s_\lambda(q,t)}, \qquad \hbox{{\rm \ie with coefficients}}\qquad \bleu{c_\lambda^{\tau}\in\N}.
   \end{equation}
 The sum runs over (length $\leq 2$) partitions  $\lambda$  such that  $|\lambda|\leq |\tau|$.
\end{conjecture}

Specific values are:
   \begin{equation}
          \bleu{\Cat_{(n-k,k)}} = \bleu{\sum_{j=0}^k s_{(n-2j,j)}},
    \end{equation}
 whenever $n> 3k-1$, and  $\Cat_{2k-1,k} = \Cat_{2k,k-1}$. These exhaust all possibilities for two part triangular partitions, since we must have $n\geq 3k-1$ if we want $\tau=(n-k,k)$ to be triangular. See~\cite[section 3]{MR4145982}

\subsection{Several parameters} Conjecturally, there is a natural extension of the previous Schur-expansions that encompass situations involving more parameters, hence involving Schur functions indexed by partitions having more than two parts. In some instances these are obtained as the  $GL_k$-character of the $\S_n$-alternating component of representations of $GL_k\times \S_n$, as discussed in \cite{2003.07402,1909.03531,MR3682396}, thus explaining the Schur positivity. In the representation theoretic framework, one can show that these expressions become stable once $k\geq n$.  We may thus present them as  positive integer coefficient linear combination of ``formal'' Schur expansions   $s_\lambda(\boldsymbol{q})=s_\lambda(q_1,q_2,\ldots,q_k)$:
\begin{equation}
   \bleu{\Cat_{\tau}(\boldsymbol{q})}:=\bleu{\sum_{|\lambda|\leq |\tau|} c_\lambda^{\tau}\, s_\lambda(\boldsymbol{q})}.
  \end{equation}
 Writing $\rho(\tau)$  for $\min(l(\tau),l(\tau'))$, we may summarize our theoretical and experimental observations about these as follows:
\begin{enumerate}
\item for all $\tau$:
   \begin{equation}
   \bleu{ \Cat_{\tau}(q,t)}=\bleu{\sum_{\alpha\subseteq \tau} t^{\dinv_{\tau}(\alpha)} q^{\area_{\tau}(\alpha))}};
  \end{equation}
\item for all $\tau$, 
 \begin{equation}\label{conjugate_property}
     \bleu{\Cat_{\tau}(\boldsymbol{q})}=\bleu{\Cat_{\tau'}(\boldsymbol{q})} ;
  \end{equation}
\item if $\lambda$ has more than $\rho(\tau)$ parts, then $c_\lambda^{\tau}=0$;
\item if $c_\lambda^{\tau}\neq 0$,  then $|\lambda|+\binom{l(\lambda)}{2}\leq |\tau|$;
\item if $\tau=\tau'+\tau_{n,n}$, with $\tau'$ triangular of length at most $n$, then 
\begin{equation}\label{nabla_property}
   \bleu{e_n^\perp \Cat_{\tau}(\boldsymbol{q})}:=\bleu{\Cat_{\tau'}(\boldsymbol{q})},
  \end{equation}
\item  for any $\tau$, 
\begin{equation}\label{delta_conjecture_un}
   \bleu{(e_{k}^\perp \Cat_{\tau})(q)}=\bleu{q^{|\tau|-\binom{k+1}{2}}\qbinom{\rho(\tau)}{k}_{{1/q}}}.
  \end{equation}
\item for all $k$ and $n$ such that $0\leq k\leq n-1$, 	and $\tau:=\tau_{n,n}$, then
\begin{equation}\label{delta_conjecture}
   \bleu{(e_{n-k-1}^\perp \Cat_{\tau})(q,t)}:=\bleu{\sum_{\alpha\subseteq \tau} t^{\dinv_{\tau}(\alpha)} \Big( \sum_{\newatop{\des(\alpha)\subseteq J }{|J|=k}}q^{\area_{\tau}^J(\alpha)} \Big)},
  \end{equation}
\end{enumerate}
where the second summation runs over subsets $J$ of $\{i\,|\, 1\leq i\leq n-1\}$ that satisfy the stated requirements.
To finish parsing the right-hand side of this last equality, we recall \autoref{define_j_area}, and set
\begin{equation}
    \bleu{\des(\alpha)}:=\bleu{\{i\,|\, 1\leq i\leq n-1,\ {\rm and}\  \alpha_i>\alpha_{i+1}\}}.
\end{equation} 
We underline that this right-hand side is exactly the combinatorial description of the coefficient of $e_n$ in the Delta theorem.
Extending this equality to any triangular partition would thus have an interesting impact on possible extensions of this theorem. 

Besides cases already covered, some experimentally calculated values are as follows
  \begin{align*}    
&\Cat_{321}=s_{111} + s_{31} + s_{41} + s_{6},\\
&\Cat_{421}=s_{211} + s_{32} + s_{41} + s_{51} + s_{7},\\
&\Cat_{431}=s_{311} + s_{42} + s_{51} + s_{61} + s_{8},\\
&\Cat_{432}= s_{411} + s_{33} +  s_{52} + s_{61} + s_{71} + s_{9},\\ 
&\Cat_{531}=  s_{221} + s_{411} + s_{42} + s_{52} + s_{61} + s_{71} + s_{9},\\
&\Cat_{532}= s_{311} + s_{521} +  s_{43} + s_{52} + s_{62} + s_{71} + s_{81} + s_{10.},\\
&\Cat_{631}= s_{311} + s_{521} +  s_{43} + s_{52} +  s_{62} + s_{71} + s_{81} + s_{10.},\\
&\Cat_{4321}= s_{1111} + s_{311} + s_{411}+ s_{511}   + s_{42} + s_{43} + s_{61} + s_{62}+ s_{71} + s_{81} + s_{10.},\\
&\Cat_{54321}=s_{11111} + s_{3111} + s_{4111} + s_{5111} + s_{6111}  + s_{611} + s_{711} + 2\,s_{811}  + s_{911} + s_{10.11}\\
   &\qquad\qquad  + s_{421} + s_{521} + s_{621}+ s_{721}  + s_{821}  + s_{431}  + s_{531} + s_{631}  + s_{441} \\
   &\qquad\qquad + s_{44} + s_{64} + s_{74}  + s_{63}  + s_{73} + s_{83} + s_{93} +  s_{72} + s_{82} + s_{92}\\
   &\qquad\qquad  + s_{10.2} + s_{10.1}  + s_{11.1} + s_{11.2} + s_{12.1} + s_{13.1} + s_{15.}.
      \end{align*}
 Observe that one may ``predict'' the equality $\Cat_{532}=\Cat_{631}$ using \autoref{conjugate_property} in conjunction with \autoref{nabla_property}, since $532=211+321$,  $631=31+321$, and $31$ is the conjugate of $211$.
Observe in $\Cat_{54321}$ that some coefficients are larger than $1$. Although these last values may be simply inferred just from the  knowledge of \autoref{delta_conjecture}, they do agree with representation theoretic descriptions not discussed here. For more on this see~\cite{MR3045143}.

\subsection{Hook shape component.}
A classical plethystic evaluation of Schur functions makes it easy to restrict a symmetric function to its ``hook shape component''. Indeed, recall that
 \begin{equation}
   \bleu{\frac{1}{u+v} s_\mu[u-\varepsilon v]}:=\bleu{\begin{cases}
     u^av^b & \text{if}\ \mu=(a\,|\,b), \\
      0 & \text{otherwise},
\end{cases}}
  \end{equation}
using Frobenius notation for hooks, and with $\varepsilon$ standing for a ``formal plethystic variable'' such that $f[\varepsilon \x ]=\omega f(\x )$. Then, it is easy to show that 
\begin{proposition} Assuming \autoref{delta_conjecture_un}, 
\begin{align}
   \bleu{\frac{1}{u+v} \Cat_{\tau}[u-\varepsilon v]}
    &=
   \bleu{\sum_{a+b=n-1}c_{(a\,|\,b)}^\tau u^av^b}\nonumber\\
    &=
   \bleu{u^{|\tau|} \prod_{i=1}^{\rho(\tau)-1} \left(1+\frac{v}{u^i}\right).}\label{formule_hook}
  \end{align}
  where $\rho(\tau):=\min(l(\tau),l(\tau'))$.
 \end{proposition}
In other words, the multiplicity of a hook indexed Schur function $s_{(a\,|\,b)}$ in $\Cat_{\tau}$ is the coefficient of $u^a v^b$ in \autoref{formule_hook}. 
For instance, the above formula gives $16$ of the $37$ terms of $\Cat_{54321}$, leaving only 21 terms to be explained:
\begin{align*}
\Cat_{54321} &= (16\ \hbox{terms}) + s_{421} + s_{521} + s_{621} + s_{721}  + s_{821} + s_{431}  + s_{531} + s_{631}  + s_{441} \\ 
 &\qquad\qquad  + s_{44} + s_{64} + s_{74}  + s_{63}  + s_{73} + s_{83} + s_{93} +  s_{72} + s_{82} + s_{92} + s_{10.2}+s_{11.2}.
      \end{align*}
 The remaining terms may be obtained using \autoref{delta_conjecture}. 
 
 A striking fact is that $\Cat_{\tau}[\boldsymbol{z}-\varepsilon a]$ essentially corresponds to the triply graded Poincar\'e series of the Khovanov-Rozansky homology of some torus knots, for adequate choices of $\tau$.
      
\section{Triangular parking functions}\label{sec_park}

 \begin{wrapfigure}{l}{0.22\textwidth}
  \vskip-10pt
   \quad
\includegraphics{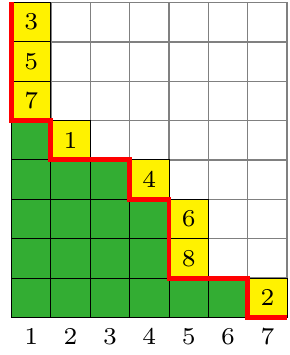}
\vskip-8pt
\wrapcaption{Parking function}
\vskip-15pt
\end{wrapfigure}
For any partition $\alpha$ and $n\geq l(\alpha)$, a height $n$ \define{parking function} of form $\alpha$ is simply a standard tableau $\pi$ of shape $(\alpha+1^n)/\alpha$. Observe that the skew-partition $(\alpha+1^n)/\alpha$ is a set of disjoint columns, say of respective size $c_i$. 
It follows readily that the number of standard tableaux considered is:
 \begin{equation}\label{dim_parking_alpha}
     \bleu{\frac{n!}{c_1!c_2!\cdots c_k!}},
\end{equation}
Observe also that the skew Schur function $s_{(\alpha+1^n)/\alpha}$ is equal to
 \begin{equation}\label{schur_parking_alpha}
     \bleu{s_{(\alpha+1^n)/\alpha}}=\bleu{e_{c_1}e_{c_2}\cdots e_{c_k}}.
\end{equation}
Given a triangular partition $\tau$ and $n\geq l(\tau)$,  the overall set  \define{$\tau$-parking functions} of height $n$, is the set $\{\pi\,|\, \pi\in\SYT((\alpha+1^n)/\alpha),\ {\rm and}\ \alpha\subseteq \tau\}$ which we denote by  $\E_{(\tau,n)}$.
For $\tau=\tau_{mn}$, with $(m,n)$ coprime integers, the total number of $\tau$-parking functions (of height $n$) is well known to be equal to $m^{n-1} $.
More generally, for $\tau=\tau_{mn}$ with $(m,n)=(da,db)$, such that $a,b$ coprime and $d=\gcd(m, n)$, we also have a parking function analog of \autoref{formule_bizley}:
\begin{equation}\label{prop_bizley}
     	   \bleu{\#\E_{(\tau_{mn},n)}}=
	   \bleu{\sum_{\lambda\vdash d} \frac{1}{z_\lambda}\binom{n}{\lambda}\,\prod_{k\in\lambda} \frac{1}{a}(ka)^{kb-1}}.
\end{equation}

\subsection{Symmetric function counting}
It is natural to extend the above enumeration to a symmetric function $q$-enumeration of height $n$ parking functions, setting:
        \begin{equation}\label{Definition_park_enumerator}
      		\bleu{\E_{(\tau,n)}(q;\x )} := \bleu{\sum_{\alpha\subseteq \tau} q^{|\tau|-|\alpha|} \, s_{(\alpha+1^n)/\alpha}(\x )}.
      \end{equation}
In the special cases $\tau=\tau_{(da,db)}$, with $a,b$ coprime, there is (see \cite{Armstrong2015,MR3882516}) a symmetric function version of \autoref{bizley_gen}:
\begin{equation}\label{bizley_gen_sym}
   \bleu{ \sum_{d=1}^\infty \E_{(\tau_{(da,db)},db)}(1;\x )\ z^d}=\bleu{ \exp\!\Big(\sum_{k\geq 1}{\textstyle \frac{1}{a}e_{kb}[ka\,\x]\, z^k/k}\Big)}.
%
\end{equation} 
We have the following nice extension of \autoref{General_q_Catalan_recurrence} for the symmetric function enumeration of parking functions, whose proof follows a very similar argument.
\begin{proposition}\label{proposition_Parking_recurrence}
For any triangular partition $\tau$ and any $n> l(\tau)$,  the generic symmetric function enumerator of $\tau$-parking functions satisfies the recurrence
\begin{equation}\label{Triangular_Parking_recurrence}
    \bleu{\E_{(\tau,n)}(q;\x )}=\bleu{q^{|\delta|}\E_{(\tau^\circ,n)}(q;\x )+\sum_{(\alpha,\beta,k)}q^{|\alpha\cap\delta|} \E_{(\alpha\setminus \partial,n-k)}(q;\x )\,\E_{(\beta,k)}(q;\x )},
\end{equation}
with initial condition $\E_{\varepsilon,n}(\x )=e_n(\x )$, and where the summation runs over the set of triples $(\alpha,\beta,k)$ such that $\tau=\alpha\odot 1\odot \beta$, with the partition $1$ corresponding to a cell $(i,k)$ of the diagonal $\partial=\partial_\tau$ of $\tau$. 
\end{proposition}
Formulas for a $q,t$-enumeration of $\tau$-parking functions with special values of $n$ may be found in~\cite{2102.07931}. One may extend these for general values of $n$, and there are stable several parameter (\ie $\q=(q_1,q_2,\ldots,q_k)$) extensions of the form:
	$$\E_{\tau,n}(\q,\x) = \sum_{\mu\vdash n}\sum_{\lambda} c_{\lambda\mu} s_\lambda(\q)\,s_\mu(\x),$$
which specialize at $\q=(q,t)$ to the above mentionned $q,t$-enumeration. This is ongoing work.

\nocite{*}

\bibliographystyle{amsplain-ac}
\bibliography{Article_Triangulaire}

\providecommand{\bysame}{\leavevmode\hbox to3em{\hrulefill}\thinspace}
\providecommand{\MR}{\relax\ifhmode\unskip\space\fi MR }
\providecommand{\MRhref}[2]{%
  \href{http://www.ams.org/mathscinet-getitem?mr=#1}{#2}
}
\providecommand{\href}[2]{#2}
\begin{thebibliography}{10}

\bibitem{ArmstrongLoehr}
Drew Armstrong, Nicholas~A. Loehr, and Gregory~S. Warrington, \emph{Rational
  parking functions and {C}atalan numbers}, Annals of Combinatorics \textbf{20}
  (2016), no.~1, 21--58, \url{https://doi.org/10.1007/s00026-015-0293-6}.

\bibitem{Armstrong2015}
Drew Armstrong, Victor Reiner, and Brendon Rhoades, \emph{Parking spaces},
  Advances in Mathematics \textbf{269} (2015), 647--706,
  \url{https://doi.org/10.1016/j.aim.2014.10.012}.

\bibitem{MR3882516}
Jean-Christophe Aval and Fran\c{c}ois Bergeron, \emph{A note on: rectangular
  {S}chr\"{o}der parking functions combinatorics}, S\'{e}m. Lothar. Combin.
  \textbf{79} ([2018--2020]), Art. B79a, 13,
  \url{https://arxiv.org/abs/arXiv:1603.09487}.

\bibitem{MR3045143}
Fran\c{c}ois Bergeron, \emph{Multivariate diagonal coinvariant spaces for
  complex reflection groups}, Adv. Math. \textbf{239} (2013), 97--108,
  \url{https://doi.org/10.1016/j.aim.2013.02.013}.

\bibitem{MR3682396}
\bysame, \emph{{O}pen questions for operators related to rectangular {C}atalan
  combinatorics}, J. Comb. \textbf{8} (2017), no.~4, 673--703,
  \url{https://doi.org/10.4310/JOC.2017.v8.n4.a6}.

\bibitem{1909.03531}
\bysame, \emph{$({G}l_k \times {S}_n)$-modules and nabla of hook-indexed
  {S}chur functions}, 2019, \url{https://arxiv.org/abs/arXiv:1909.03531}.

\bibitem{2003.07402}
\bysame, \emph{$({G}l_k\times {S}_n)$-modules of multivariate diagonal
  harmonics}, 2020, \url{https://arxiv.org/abs/arXiv:2003.07402}.

\bibitem{MR61567}
M.~T.~L. Bizley, \emph{Derivation of a new formula for the number of minimal
  lattice paths from {$(0,0)$} to {$(km,kn)$} having just {$t$} contacts with
  the line {$my=nx$} and having no points above this line; and a proof of
  {G}rossman's formula for the number of paths which may touch but do not rise
  above this line}, J. Inst. Actuar. \textbf{80} (1954), 55--62.

\bibitem{2102.07931}
Jonah Blasiak, Mark Haiman, Jennifer Morse, Anna Pun, and George~H. Seelinger,
  \emph{A shuffle theorem for paths under any line}, 2021,
  \url{https://arxiv.org/abs/arXiv:2102.07931}.

\bibitem{MR1692980}
Sylvie Corteel, Ga\"{e}l R\'{e}mond, Gilles Schaeffer, and Hugh Thomas,
  \emph{The number of plane corner cuts}, Adv. in Appl. Math. \textbf{23}
  (1999), no.~1, 49--53, \url{https://doi.org/10.1006/aama.1999.0646}.

\bibitem{MR4145982}
Eugene Gorsky, Graham Hawkes, Anne Schilling, and Julianne Rainbolt,
  \emph{Generalized {$q, t$}-{C}atalan numbers}, Algebr. Comb. \textbf{3}
  (2020), no.~4, 855--886, \url{https://doi.org/10.5802/alco.120}.

\bibitem{MR3811519}
James Haglund, Jeffrey~Brian Remmel, and Andrew~Timothy Wilson, \emph{The
  {D}elta conjecture}, Trans. Amer. Math. Soc. \textbf{370} (2018), no.~6,
  4029--4057, \url{https://doi.org/10.1090/tran/7096}.

\bibitem{MR1905123}
M.~Lothaire, \emph{Algebraic combinatorics on words}, Encyclopedia of
  {M}athematics and its {A}pplications, vol.~90, Cambridge University Press,
  Cambridge, 2002, \url{https://doi.org/10.1017/CBO9781107326019}.

\bibitem{MR2017036}
Irene M\"{u}ller, \emph{Corner cuts and their polytopes}, Beitr\"{a}ge Algebra
  Geom. \textbf{44} (2003), no.~2, 323--333.

\bibitem{MR1692984}
Shmuel Onn and Bernd Sturmfels, \emph{Cutting corners}, Adv. in Appl. Math.
  \textbf{23} (1999), no.~1, 29--48,
  \url{https://doi.org/10.1006/aama.1999.0645}.

\bibitem{MR0429672}
C.R. Platt, \emph{Planar lattices and planar graphs}, J. Combinatorial Theory
  Ser. B \textbf{21} (1976), no.~1, 30--39.

\end{thebibliography}

\end{document}